\documentclass[11pt]{amsart}
\usepackage{amsmath}
\usepackage{amssymb}
\usepackage{amsthm}
\usepackage{latexsym}
\usepackage{graphicx}
\usepackage{enumerate}
\usepackage[all]{xy}
\usepackage{chngcntr}
\usepackage[colorlinks=true,linkcolor=blue,urlcolor=blue,citecolor=red]{hyperref}

\newtheorem{thm}{Theorem}[section]

\newtheorem{lem}[thm]{Lemma}
\newtheorem{prop}[thm]{Proposition}

\newtheorem{thmintro}{Theorem}

\theoremstyle{definition}
\newtheorem{defn}[thm]{Definition}
\newtheorem{rem}[thm]{Remark}

\newcommand{\Z}{\mathbb Z}
\newcommand{\Q}{\mathbb Q}
\newcommand{\R}{\mathbb R}
\newcommand{\C}{\mathbb C}

\newcommand{\mc}{\mathcal}

\def\Irr{{\rm Irr}}
\newcommand{\mr}{\mathrm}

\newcommand{\enuma}[1]{\begin{enumerate}[\textup{(}a\textup{)}] {#1} \end{enumerate}}

\newcommand{\nr}{\mathrm{nr}}

\newcommand{\Rep}{\mathrm{Rep}}

\newcommand{\der}{\mathrm{der}}

\newcommand{\Hom}{\mathrm{Hom}}
\newcommand{\End}{\mathrm{End}}

\newcommand{\Gal}{\mathrm{Gal}}

\begin{document}

\title[Standard modules and intertwining operators for $p$-adic groups]
{Standard modules and intertwining operators\\ for reductive $p$-adic groups}
\date{\today}
\maketitle
\vspace{4mm}

\begin{center}
{\Large Maarten Solleveld} \\[1mm]
IMAPP, Radboud Universiteit Nijmegen\\
Heyendaalseweg 135, 6525AJ Nijmegen, the Netherlands \\
email: m.solleveld@science.ru.nl
\end{center}
\vspace{4mm}

\begin{abstract}
Consider a reductive group $G$ over a non-archimedean local field. The Galois
group Gal$(\C/\Q)$ acts naturally on the category of smooth complex 
$G$-representations. We prove that this action stabilizes the class of 
standard $\C G$-modules. This generalizes and relies on an analogous result
from \cite{KSV} about essentially square-integrable representations.

Other important objects in the proof of our main result are intertwining
operators between parabolically induced $G$-representations, and the 
associated Harish-Chandra $\mu$-functions. We determine an explicit
formula for the $\mu$-function of any irreducible representation of
any Levi subgroup of $G$.
\end{abstract}
\vspace{4mm}

\tableofcontents

\section{Introduction}

This paper is a sequel to \cite{KSV}. That project started with the question:
which classes of representations of reductive $p$-adic groups $G$ are stable
under the action of $\Gal (\C / \Q)$? By default, the representations that we
consider here are smooth and on complex vector spaces. The motivation for such
questions is twofold.

Firstly, it relates to L-functions. One may hope to prove statements of the
kind 
\[
\text{if } L(s,\pi) = 0 \text{ for some } s \in \frac{1}{2} \Z \text{, then }
L (s, \gamma \cdot \pi) = 0 \text{ for } \gamma \in \Gal (\C / \Q).
\]
This could apply to representations $\pi$ of reductive groups over local fields
or of adelic reductive groups (and of course one needs reductive groups for
which L-functions of irreducible representations are defined). For general
linear groups, this has been studied in \cite{KrCl}.

Secondly, in algebro-geometric investigations related to reductive $p$-adic 
groups it is often beneficial to use representations not over $\C$ but over 
$\overline{\Q_\ell}$ for a prime number $\ell \neq p$. Here we are thinking in
particular of the Fargues--Scholze program \cite{FaSc}, of the generalized
Springer correspondence \cite{Lus1} and of geometric graded Hecke algebras
\cite{AMS}. One may wonder whether certain results about $\C$-representations
obtained via $\overline{\Q_\ell}$-representations depend on $\ell$ or on the 
choice of a field isomorphism $\C \cong \overline{\Q_\ell}$. Any two such 
field isomorphisms differ by composition with an element of $\Gal (\C / \Q)$, 
so one wants to understand which properties of $\C$-representations
preserved by $\Gal (\C / \Q)$.

It is clear that the action of $\Gal (\C / \Q)$ on $G$-representations
preserves irreduci\-bi\-li\-ty, and it is easy to see that it preserve cuspidality.
However, this action does in ge\-ne\-ral not preserve analytic notions like
unitarity, temperedness or square-integrability modulo center. The main results
of \cite{KSV} say that $\Gal (\C / \Q)$ stabilizes
\begin{itemize}
\item the class of essentially square-integrable $G$-representations,
\item the class of elliptic (virtual) $G$-representations.
\end{itemize}
In this paper we focus on a larger class of representations, that of
standard $\C G$-modules.
Let $Q = M U_Q$ be a parabolic subgroup of $G$ and let $\tau$ be an irreducible
tempered $M$-representation. Let $\nu \in \Hom (M, \R_{>0})$ be strictly positive 
with respect to $Q$ (for the precise condition see Section \ref{sec:2}). 
By definition, a standard $\C G$-module is a $G$-representation of the form 
$I_Q^G (\tau \otimes \nu)$, with $(Q,\tau,\nu)$ as above.
The importance of standard modules stems from the Langlands classification
(which for $p$-adic groups is not due to Langlands):

\begin{thmintro}\label{thm:7.1} \textup{\cite[\S VII.4]{Ren}}
\enuma{
\item Every standard $\C G$-module $I_Q^G (\tau \otimes \nu)$ has a unique
irreducible quotient, which we call $\mc L (Q,\tau \otimes \nu)$.
\item Every irreducible $G$-representation $\pi$ arises as the quotient of
a standard $\C G$-module $\pi_{st}$.
\item If $I_{Q'}^G (\tau' \otimes \nu')$ is a standard module and
$\mc L (Q,\tau \otimes \nu) \cong \mc L (Q',\tau' \otimes \nu')$, then
there exists a $g \in G$ such that $g Q g^{-1} = Q'$, $g M g^{-1} = M'$
and $g (\tau \otimes \nu) \cong \tau' \otimes \nu'$.
\item The maps $I_Q^G (\tau \otimes \nu) \mapsto \mc L (Q, \tau \otimes \nu)$
and $\pi \mapsto \pi_{st}$ set up a bijection between $\Irr (\C G )$ and the
set of standard $\C G$-modules (up to isomorphism).
\item The set of standard $\C G$-modules (up to isomorphism) forms a $\Z$-basis of
the Grothendieck group of the category of finite length $G$-representations.
}
\end{thmintro}

It is expected that in categorical versions of the local Langlands correspondence,
standard modules behave better than irreducible $G$-representations. The reason
should be that non-elliptic standard modules always come in families (because $\nu$ 
can vary continuously), which does not hold for irreducible representations.

\subsection{Main results}

\begin{thmintro}\label{thm:B}
The action of $\Gal (\C / \Q)$ on the category of smooth complex 
$G$-repre\-sen\-ta\-tions stabilizes the class of standard $\C G$-modules.
\end{thmintro}

Theorem \ref{thm:B} enables us to define standard $\overline{\Q_\ell} G$-modules
in an ambiguous way. Namely, we call a $G$-representation $\pi_\ell$ on a 
$\overline{\Q_\ell}$-vector space standard if the complex $G$-representation
obtained from $\pi_\ell$ via any field isomorphism $\C \cong \overline{\Q_\ell}$
is a standard $\C G$-module.

Essential ingredients for Theorem \ref{thm:B} are Harish-Chandra's
intertwining operators
\[
J_{P'|P}(\pi) : I_P^G (\pi) \to I_{P'}^G (\pi)
\quad \text{for finite length } L\text{-representations } \pi .
\]
In fact we need more properties than can be found in the literature, so we
further develop the theory of intertwining operators. Let $\pi \in \Irr (L)$
be an irreducible $L$-representation. The invertibility of $J_{P'|P}(\pi)$
is governed by Harish-Chandra's $\mu$-function $\mu_{G,L}(\pi)$ \cite{Wal}.
More precisely, $\mu_{G,L}(\pi \otimes \chi)$ is a rational function of an
unramified character $\chi \in X_\nr (L)$, and $J_{P'|P}(\pi)$ is invertible
if $\mu_{G,L}(\pi) \in \C^\times$. Usually $J_{P'|P}(\pi)$ is not invertible
if $\mu_{G,L}(\pi) \in \{0,\infty\}$.

\begin{thmintro}\textup{(see Proposition \ref{prop:1.2} and Theorem \ref{thm:1.4})}
\label{thm:C} \\
Let $M \subset L$ be a Levi subgroup and let $\sigma \in \Irr (M)$ be such that
$\pi \in \Irr (L)$ is a subquotient of $I_{MU}^L (\sigma)$, for some parabolic
subgroup $MU$ of $L$. 
\enuma{
\item There exists an explicit $c \in \R_{>0}$ such that
\[
\mu_{G,L}(\pi \otimes \chi) = c \, \mu_{G,M}(\sigma \otimes \chi) \mu_{L,M}(\sigma 
\otimes \chi)^{-1} \qquad \chi \in X_\nr (L).
\]
\item Suppose in addition that $\sigma$ is cuspidal. Then
\[
\mu_{G,L}(\pi \otimes \chi) = c \, \prod\nolimits_{M_\alpha} \mu_{M_\alpha,M}(\sigma 
\otimes \chi) \qquad \chi \in X_\nr (L),
\]
where the product runs over the Levi subgroups $M_\alpha \subset G$ which contain
$M$ as minimal Levi subgroup but are not contained in $L$. Moreover each term
$\mu_{M_\alpha,M}(\sigma \otimes \chi)$ admits an explicit formula as a rational
function of $\chi$.
}
\end{thmintro}

\subsection{Structure of the main proof} \

The initial step towards Theorem \ref{thm:B} is an alternative construction of
standard modules, from \cite{Sol1}. Let $P = L U_P$ be a parabolic subgroup of
$G$ and let $\delta$ be an irreducible essentially square-integrable 
$L$-representation. We say that $\delta$ is positive with respect to $P$ if
the absolute value of the central character of $\delta$ is so. In that case
$I_P^G (\delta)$ is a direct sum of standard $\C G$-modules $I_P^G (\delta)_\kappa$. 
(See Paragraph \ref{par:1.5} for meaning of $\kappa$.)
Moreover every standard $\C G$-module arises in this way, from essentially 
unique $(P,L,\delta)$.

Without the positivity condition on $\delta$, $I_P^G (\delta)$ is a direct sum
of so-called quasi-standard $\C G$-modules $I_P^G (\delta )_\kappa$ (Definition 
\ref{def:7.1}). Any quasi-standard $\C G$-module $I_P^G (\delta )_\kappa$
can be made into a standard $\C G$-module by adjusting $P$, but in general that 
changes the isomorphism class of the module. Since $\Gal (\C / \Q)$ preserves
essential square-integrability \cite{KSV}, $\Gal (\C / \Q)$ stabilizes the class
of quasi-standard $\C G$-modules (Lemma \ref{lem:7.9}).

From this point on we present two proofs of Theorem \ref{thm:B}, both of 
interest in their own way. The first method relies on an invariant $\mc N$
of $G$-representations $\pi$, which measures a distance from $\pi$ to the set
of parabolic inductions of unitary cuspidal representations of Levi subgroups
of $G$ (see Paragraph \ref{par:2.2}). It is known
from \cite{Sol1} that $\mc L (Q,\tau \otimes \nu)$ is the unique irreducible
subquotient of $I_Q^G (\tau \otimes \nu)$ which has the same $\mc N$-value as
$I_Q^G (\tau \otimes \nu)$. This enables us to characterize standard
$\C G$-modules as those quasi-standard $\C G$-modules which have an
irreducible quotient with the appropriate $\mc N$-value (Theorem \ref{thm:7.6}).
In contrast to the original definition, this characterization of standard
modules uses neither temperedness not positivity of characters.

We show that this configuration is preserved when we let any $\gamma \in 
\Gal (\C / \Q)$ act on a standard $\C G$-module. That leads to our first
proof of Theorem \ref{thm:B}, in Proposition \ref{prop:3.3}. However, this
proof is conditional: we assume that $\gamma$ preserves the $\mc N$-values
of all essentially square-integrable representations of Levi subgroups of $G$.
That property is not yet known, but it follows from the rationality of 
$q$-parameters for related Hecke algebras. Such rationality has been 
conjectured by Lusztig \cite{Lus2}, and has been checked in the large
majority of all cases \cite{SolParam,Oha}.

Our second proof of Theorem \ref{thm:B} uses that the parabolic subgroup $P$
in a quasi-standard $\C G$-module $I_P^G (\delta )_\kappa$ is often not unique.
Namely, for any other parabolic subgroup $P'$ with the same Levi factor $L$,
there exists an intertwining operator
\begin{equation}\label{eq:1}
J_{P'|P}(\delta) : I_P^G (\delta) \to I_{P'}^G (\delta) .
\end{equation}
Under mild conditions \eqref{eq:1} is an isomorphism, which entails that
$I_P^G (\delta )_\kappa$ is isomorphic to a quasi-standard direct summand
of $I_{P'}^G (\delta)$. 

With Theorem \ref{thm:C} one can reduce questions about intertwining operators
and $\mu$-functions to the cases of cuspidal representations, which can be 
analysed more easily. For instance, consider the corank one intertwining
operator $J_{M U_{-\alpha} | M U_\alpha}(\sigma \otimes \chi)$, where
$M U_{-\alpha}$ and $M U_\alpha$ are the parabolic subgroups of $M_\alpha$
with Levi factor $M$. It was already known that, if $\mu_{M_\alpha,M}(\sigma) = 0$,
then $J_{M U_{-\alpha} | M U_\alpha}(\sigma \otimes \chi)$ can be normalized
to an operator 
\[
J'_{M U_{-\alpha} | M U_\alpha}(\sigma \otimes \chi) : I_P^G (\sigma \otimes \chi)
\to I_{P'}^G (\sigma \otimes \chi) ,
\]
which is invertible for $\chi$ in a neighborhood of 1 in $X_\nr (M)$. In particular
$I_{M U_{-\alpha}}^{M_\alpha}(\sigma)$ is isomorphic to $I_{M U_\alpha}^{M_\alpha}
(\sigma)$ whenever $\mu_{M_\alpha,M}(\sigma) \neq \infty$. More generally, we
prove in Theorem \ref{thm:1.7} that (with the notations from Theorem \ref{thm:C})
\begin{equation}\label{eq:2}
I_P^G (\pi) \cong I_{P'}^G (\pi) \text{ unless } \mu_{M_\alpha,M}(\sigma) = \infty
\text{ for some } M_\alpha \text{ with } P \supset M U_\alpha \not\subset P' .
\end{equation}
This is used in our second proof of Theorem \ref{thm:B}. 

For a given $\gamma \in \Gal (\C / \Q)$ and 
$(P,L,\delta)$ as above, we construct a particular $P' = L U_{P'}$ such that
$\gamma \cdot \delta$ is positive with respect to $P'$. An explicit analysis of 
the corank one situation (Proposition \ref{prop:2.1}) reveals an asymmetry
between $I_{M U_\alpha}^{M_\alpha}(\sigma)$ and 
$I_{M U_{-\alpha}}^{M_\alpha}(\sigma)$ when $\mu_{M_\alpha,M}(\sigma) = \infty$,
the roles of the unique quotient and the unique subrepresentation differ.
Using that with as $\sigma$ a representative of the cuspidal support of $\delta$, 
we can arrange that all the $M_\alpha$ with 
$\mu_{M_\alpha,M}(\sigma) = \infty$ satisfy $M U_\alpha \subset P \cap P'$
or $M U_{-\alpha} \subset P \cap P'$. With \eqref{eq:2}, it follows that
a normalized version $J'_{P'|P}(\delta)$ of $J_{P'|P}(\delta)$ gives
isomorphisms 
\begin{equation}\label{eq:3}
I_P^G (\delta) \cong I_{P'}^G (\delta) \quad \text{and} \quad
\gamma \cdot I_P^G (\delta) \cong \gamma \cdot I_{P'}^G (\delta) .
\end{equation}
From there, we show in Theorem \ref{thm:7.11} that 
$\gamma \cdot I_P^G (\delta )_\kappa$ is a standard $\C G$-module.

\renewcommand{\theequation}{\arabic{section}.\arabic{equation}}
\counterwithin*{equation}{section}

\section{Notations} 

$F$: non-archimedean local field

$G$: $F$-rational points of a connected reductive group $\mc G$ defined over $F$

$G_\der$: $F$-rational points of derived subgroup of $\mc G$

$Z (G)$: center of $G$

$A_G$: maximal $F$-split torus in $Z(G)$

$\Rep (G)$: category of smooth complex $G$-representations

$\Irr (G)$: set of irreducible objects in $\Rep (G)$, up to isomorphism

$G^1$: subgroup of $G$ generated by all compact subgroups of $G$

$X_\nr (G) = \Hom (G/G^1,\C^\times)$: group of unramified characters of $G$

$L,M$: Levi subgroups of $G$

$P = L U_P$: parabolic subgroup of $G$ with Levi factor $L$ and unipotent radical $U_P$

$\overline{P} = L U_{\overline P}$: parabolic subgroup opposite to $P = L U_P$

$I_P^G$: normalized parabolic induction functor $\Rep (L) \to \Rep (G)$

\section{Intertwining operators and Harish-Chandra's $\mu$-functions} \label{sec:1}

We recall the definition of Harish-Chandra's intertwining operators. Consider two 
parabolic subgroups $P = L U_P$ and $P' = L U_{P'}$ with a common Levi 
factor $L$. Let $(\pi,V_\pi)$ be a $L$-representation. All the representations 
$I_P^G (\pi \otimes \chi)$ with $\chi \in X_\nr (L)$ can be realized on the same 
vector space, namely $\mathrm{ind}_{P \cap K_0}^{K_0} V_\pi$ for a good maximal compact
subgroup $K_0$ of $G$. This makes it possible to speak of objects on 
$I_P^G (\pi \otimes \chi)$ that vary regularly or
rationally as functions of $\chi \in X_\nr (L)$. Consider the intertwining operators
\begin{equation}\label{eq:6.1}
\begin{array}{llll}
J_{P' |P}(\pi \otimes \chi) : & I_P^G (\pi \otimes \chi) & \to & 
I_{P'}^G (\pi \otimes \chi) \\
 & f & \mapsto & \big[ g \mapsto \int_{U_P \cap U_{P'} \setminus U_{P'}} 
 f(ug) \, \textup{d} u \big]
\end{array}.
\end{equation}
For $\pi$ of finite length, this is well-defined as a family of $G$-homomorphisms depending
rationally on $\chi \in X_\nr (L)$ \cite[Th\'eor\`eme IV.1.1]{Wal}. There is an
alternative construction of \eqref{eq:6.1}, in \cite[proof of Th\'eor\`eme IV.1.1]{Wal}.
That construction works for representations with coefficients in any algebraically 
closed field of characteristic not~$p$, which has been exploited recently in \cite{MoTr}
to define intertwining operators in more general settings.

Let $\bar P = L U_{\bar P}$ be the parabolic subgroup opposite to $P = L U_P$.
We assume that $\pi$ is irreducible and we consider the composition 
\begin{equation}\label{eq:6.9}
J_{P | \bar P}(\pi \otimes \chi) J_{\bar P | P}(\pi \otimes \chi) :
I_P^G (\pi \otimes \chi) \to I_P^G (\pi \otimes \chi) \qquad \chi \in X_\nr (L) . 
\end{equation}
This depends rationally on $\chi$, and for generic $\chi$ the representation 
$I_P^G (\pi \otimes \chi)$ is irreducible \cite[Th\'eor\`eme 3.2]{Sau}. Therefore
\eqref{eq:6.9} is a scalar operator \cite[\S IV.3]{Wal}, say 
\begin{equation}\label{eq:6.36}
j_{G,L}(\pi \otimes \chi) \, \mathrm{id} \quad \text{with} \quad 
j_{G,L} : X_\nr (L) \pi \to \C \cup \{\infty\} .
\end{equation}
For purposes of harmonic analysis, the reciprocal of $j_{G,L}$ is often more 
convenient than $j_{G,L}$ itself. It usually rescaled by numbers
$\gamma (G|L), c(G|L) \in \Q_{>0}$ defined in \cite[p. 241]{Wal}. By 
definition \cite[\S V.2]{Wal} Harish-Chandra's $\mu$-function is
\begin{equation}\label{eq:1.3}
\mu_{G,L}(\pi \otimes \chi) = c(G|L)^2 \gamma (G|L)^2 j_{G,L}(\pi \otimes \chi )^{-1}.
\end{equation}
This $\mu$-function is especially important for essentially square-integrable
representations $\pi$, because then $\mu_{G,L}(\pi \otimes \chi)$ describes how the
Plancherel density on $\{ I_P^G (\pi \otimes \chi) : \chi \in X_\nr (L) \}$ varies
as a function of $\chi$ \cite{Wal}.

Let $A_L$ be the maximal split torus in $Z(L)$. The set of nonzero weights by which 
$A_L$ acts on the Lie algebra of $G$ is not necessarily a root system, but it is
always a generalized root system in the sense of \cite{DiFi}. In particular notions
like basis, positive roots and reduced roots still make sense. Let $\Phi (G,A_L)$ 
be the set of reduced roots of $(G,A_L)$ and let $\Phi (G,A_L)^+ = \Phi (U_P,A_L)$ 
be the subset of roots appearing in the Lie algebra of $P$. 
For $\alpha \in \Phi (G,A_L)^+$, let 
$U_\alpha$ (resp. $U_{-\alpha}$) be the root subgroup of $G$ for all positive
(resp. negative) multiples of $\alpha$. Let $L_\alpha$ be the Levi subgroup of
$G$ generated by $L \cup U_\alpha \cup U_{-\alpha}$. Then $L$ is a maximal proper
Levi subgroup of $L_\alpha$. Now \cite[IV.3.(5) and Lemma V.2.1]{Wal} say that
\begin{equation}\label{eq:1.4}
\begin{aligned}
& j_{G,L}(\pi) = \prod\nolimits_{\alpha \in \Phi (G,A_L)^+} j_{L_\alpha,L} (\pi) ,\\
& \mu_{G,L}(\pi) = \prod\nolimits_{\alpha \in \Phi (G,A_L )^+} \mu_{L_\alpha,L} (\pi). 
\end{aligned}
\end{equation}
With these $\mu$-functions one can check easily whether certain intertwining operators
are invertible.

\begin{lem}\label{lem:1.5}
Suppose that $\mu_{L_\alpha,L}(\pi) \notin \{0,\infty\}$ (or equivalently
$j_{L_\alpha,L}(\pi) \notin \{0,\infty\}$) for all
$\alpha \in \Phi (U_P,A_L) \cap \Phi (U_{\overline{P'}},A_L)$. Then
$J_{P'|P}(\pi) : I_P^G (\pi) \to I_{P'}^G (\pi)$ is invertible.
\end{lem}
\begin{proof}
As noticed on \cite[p. 279]{Wal}, there exists a sequence of parabolic subgroups
$P = P_0, P_1, \cdots , P_d = P'$, all with Levi factor $L$, such that
$\Phi (P_i,A_L)$ and $\Phi (P_{i-1},A_L)$ differ by only one root and
$d = |\Phi (U_P,A_L) \cap \Phi (U_{\overline{P'}},A_L)|$. In this situation
\cite[IV.1.(12)]{Wal} says that
\begin{equation}\label{eq:1.9}
J_{P'|P}(\pi) = J_{P_d | P_{d-1}}(\pi) \circ \cdots \circ J_{P_1|P_0}(\pi) .
\end{equation}
It suffices to show that each $J_{P_i | P_{i-1}}(\pi)$ is invertible

Therefore we may assume that $\Phi (U_{P'},A_L) \cap \Phi (U_{\overline{P'}},A_L)$ 
consists of a single root $\alpha$. By \cite[IV.1.(14)]{Wal} we may identify
\begin{equation}\label{eq:1.10}
J_{P',P}(\pi) = I_{L_\alpha P}^G \big( J_{L U_{-\alpha} |
L U_\alpha} (\pi ) \big) : I_{L_\alpha P}^G ( I_{L U_\alpha}^L (\pi)) \to
I_{L_\alpha P}^G (I_{L U_{-\alpha}}^L (\pi)) .
\end{equation}
By assumption
\[
J_{L U_{\alpha} | L U_{-\alpha}} (\pi ) J_{L U_{-\alpha} | L U_\alpha} (\pi ) =
j_{L_\alpha,L} (\pi) \: \mr{id} \in \C^\times \: \mr{id} .
\]
Hence $J_{L U_{-\alpha} | L U_\alpha} (\pi)$ is invertible and \eqref{eq:1.10}
is invertible as well.
\end{proof}

\subsection{Silberger's formulas for the $\mu$-functions} \

In \cite{Sil3,Sil4} the functions $\mu_{G,L}(\pi)$ were determined, for essentially 
square-in\-te\-gra\-ble representations. We will provide a different argument to
arrive at the same formula in larger generality.

Let $M$ be a Levi subgroup of $L$ and $Q = M U_Q$ be a parabolic subgroup of $G$
with Levi factor $M$, such that $Q \subset P$. Then $Q \cap L$ is a parabolic
subgroup of $L$ with Levi factor $M$, $P = Q L$ and $\overline{P} = L \overline{Q}$.
We note that, since $P = L \ltimes U_P$:
\begin{equation}\label{eq:1.5}
U_Q = U_{Q \cap L} \ltimes U_P \qquad \text{and} \qquad 
U_{\overline{Q}} = U_{\overline Q \cap L} \ltimes U_{\overline P} .
\end{equation}

\begin{lem}\label{lem:1.1}
Suppose that $\sigma \in \Irr (M)$ and that $\pi$ is a subquotient of 
$I_{Q \cap L}^L (\sigma)$. Then $\mu_{G,L} \big( I_{Q \cap L}^L (\sigma) \otimes 
\chi \big)$ is defined for $\chi \in X_\nr (L)$,
and equals $\mu_{G,L}(\pi \otimes \chi)$.
\end{lem}
\begin{proof}
There is a natural isomorphism $I_{Q \cap L}^L (\sigma) \otimes \chi \cong 
I_{Q \cap L}^L (\sigma \otimes \chi|_M)$. For $\chi' \in X_\nr (M)$ in generic 
position, $I_Q^L (\chi')$ is irreducible \cite[Th\'eor\`eme 3.2]{Sau}. In the same 
way as in \eqref{eq:6.9} and \eqref{eq:6.36} we see that 
\begin{equation}\label{eq:1.1}
J_{P | \bar P} \big( I_{Q \cap L}^L (\sigma \otimes \chi') \big) J_{\bar P | P}\big( I_{Q \cap L}^L (\sigma \otimes \chi') \big)
= j_{G,L} \big( I_{Q \cap L}^L (\sigma \otimes \chi') \big) \mr{id} \quad \chi' \in X_\nr (M).
\end{equation}
This shows that $j_{G,L} \big( I_{Q \cap L}^L (\sigma \otimes \chi') \big)$ and $\mu_{G,L}(I_{Q \cap L}^L \big( \sigma
\otimes \chi') \big)$ are well-defined. We note that the formulas for 
$J_{\bar P | P} \big( I_{Q \cap L}^L (\sigma \otimes \chi) \big)$ and $J_{\bar P | P}(\pi \otimes \chi)$
are essentially the same, only applied to different representations. 

Write $\pi = \pi_1 / \pi_2$ where $\pi_1, \pi_2$ are subrepresentations of 
$I_{Q \cap L}^L (\sigma)$. One can obtain $J_{\bar P|P}(\pi) : I_P^G  (\pi) \to
I_{\overline P}^G (\pi)$ from $J_{\bar P | P} \big( I_{Q \cap L}^L (\sigma \otimes \chi) \big)$ 
by first restriction to $J_{\overline{P} | P} (\pi_1)$ and then taking the induced
homomorphism on $I_P^G (\pi) \cong I_P^G (\pi_1) / I_P^G (\pi_2)$. Since
\eqref{eq:1.1} with $\chi' = \chi \in X_\nr (L)$ is a scalar operator, it follows
that $J_{P | \bar P}(\pi \otimes \chi) J_{\bar P | P}(\pi \otimes \chi)$ is also
a scalar operator, with the same scalar. In other words,
\begin{equation}\label{eq:1.2}
j_{G,L} \big( I_{Q \cap L}^L (\sigma) \otimes \chi \big) = j_{G,L}(\pi \otimes \chi) .
\end{equation}
This argument applies initially for every $\chi \in X_\nr (L)$ such that
$j_{G,L}(\pi \otimes \chi) \neq \infty$, and then it extends to all
$\chi \in X_\nr (L)$ because both $j$-functions are rational in $\chi$.
From \eqref{eq:1.2} and \eqref{eq:1.3} we see that
$\mu_{G,L} \big( I_{Q \cap L}^L (\sigma) \otimes \chi \big) = \mu_{G,L}(\pi \otimes \chi)$.
\end{proof}

The following result generalizes \cite[Theorem 1]{Sil4}.

\begin{prop}\label{prop:1.2}
In the setting of Lemma \ref{lem:1.1} we have, for $\chi \in X_\nr (L)$:
\enuma{
\item $j_{G,L}(\pi \otimes \chi) = j_{G,M}(\sigma \otimes \chi) j_{L,M}(\sigma \otimes
\chi)^{-1}$,
\item $\mu_{G,L}(\pi \otimes \chi) = {\displaystyle \frac{\mu_{G,M}(\sigma \otimes 
\chi)}{\mu_{L,M}(\sigma \otimes\chi)} \frac{c(G|L)^2 c(L|M)^2}{c(G|M)^2}}$.
}
\end{prop}
\begin{proof}
(a) In view of \eqref{eq:1.2}, we may replace $\pi \otimes \chi$ by 
$I_{Q \cap L}^L (\sigma) \otimes \chi \cong I_{Q \cap L}^L (\sigma \otimes \chi)$. 
Then all the involved expressions are defined for any $\chi \in X_\nr (M)$. 

Consider the operator
\begin{equation}\label{eq:1.6}
I_{\overline P}^G (J_{\overline {Q \cap L} | {Q \cap L}} (\sigma \otimes \chi)) \circ
J_{\overline P | P} \big( I_{Q \cap L}^L (\sigma \otimes \chi) \big) : 
I_P^G \big( I_{Q \cap L}^L (\sigma \otimes \chi) \big) \to 
%I_{\overline{P}}^G \big( I_Q^L (\sigma \otimes \chi) \big) \to 
I_{\overline{P}}^G \big( I_{\overline{Q} \cap L}^L (\sigma \otimes \chi) \big) .
\end{equation}
For $u \in G$ and a function $f$ on $G$ we write $(\lambda_u f)(g) = f (u^{-1}g)$.
Then the effect of \eqref{eq:1.6} is 
\[
f \mapsto \int_{U_{\overline P}} (\lambda_{u_1} f) \, \textup{d} u_1 \mapsto
\int_{U_{\overline{Q} \cap L}} \int_{U_{\overline P}} \lambda_{u_2} 
(\lambda_{u_1} f) \, \textup{d} u_1 \textup{d} u_2 .
\]
By \eqref{eq:1.5}, that is the same as $f \mapsto \int_{U_{\overline Q}} 
(\lambda_{u_3} f) \textup{d} u_3$. The transitivity of parabolic induction 
\cite[Lemme VI.1.4]{Ren} says that there are natural isomorphisms
\begin{equation}
I_P^G \big( I_{Q \cap L}^L (\sigma \otimes \chi) \big) \cong I_{Q}^G (\sigma \otimes \chi) \quad
\text{and} \quad I_{\overline{P}}^G \big( I_{\overline{Q} \cap L}^L (\sigma \otimes \chi) \big) 
\cong I_{\overline{Q}}^G (\sigma \otimes \chi) .
\end{equation}
Therefore \eqref{eq:1.6} can be identified with 
\[
J_{\overline{Q} | Q} (\sigma \otimes \chi) : I_{Q}^G (\sigma \otimes \chi)
\to I_{\overline{Q}}^G (\sigma \otimes \chi) .
\]
In the same way one can check that
\begin{equation}\label{eq:1.7}
J_{P | \overline P} \big( I_{Q \cap L}^L (\sigma \otimes \chi) \big) \circ 
I_{\overline P}^G (J_{Q | \overline Q} (\sigma \otimes \chi)) = 
J_{Q | \overline{Q}} (\sigma \otimes \chi) : I_{\overline{Q}}^G (\sigma \otimes
\chi) \to I_{Q}^G (\sigma \otimes \chi) .
\end{equation}
Combining \eqref{eq:1.7} and the two expressions for \eqref{eq:1.6}, we compute
\begin{align*}
j_{G,M}(\sigma \otimes \chi) \mr{id} & = J_{Q | \overline{Q}} (\sigma \otimes \chi) 
J_{\overline{Q} | Q} (\sigma \otimes \chi) = \\
J_{P | \overline P} (I_{Q \cap L}^L (\sigma \otimes & \chi)) 
I_{\overline P}^G (J_{Q \cap L | \overline Q \cap L} (\sigma \otimes \chi)) 
I_{\overline P}^G (J_{\overline Q \cap L | Q \cap L} (\sigma \otimes \chi)) 
J_{\overline P | P} \big( I_{Q \cap L}^L (\sigma \otimes \chi) \big) \\
& = J_{P | \overline P} \big( I_{Q \cap L}^L (\sigma \otimes \chi) \big) 
I_{\overline P}^G \big( j_{L,M} (\sigma \otimes \chi) \mr{id} \big) 
J_{\overline P | P} \big( I_{Q \cap L}^L (\sigma \otimes \chi) \big) \\
& = j_{L,M} (\sigma \otimes \chi) J_{P | \overline P} \big( I_{Q \cap L}^L (\sigma \otimes 
\chi) \big) J_{\overline P | P} \big( I_{Q \cap L}^L (\sigma \otimes \chi) \big) \\
& = j_{L,M} (\sigma \otimes \chi) j_{G,L} 
\big( I_{Q \cap L}^L (\sigma \otimes \chi) \big) \mr{id} .
\end{align*}
(b) Recall that $\mu_{G,L} = c(G|L)^2 \gamma (G|L)^2 j_{G,L}^{-1}$. It follows from
\cite[p. 241]{Wal} that $\gamma (G|L) = \gamma (G|M) \gamma (M|L)$, but $c(G|L)$ need
not satisfy such a relation. Thus Lemma \ref{lem:1.1} and part (a) entail
\[
\frac{\mu_{G,L}(\pi \otimes \chi)}{c(G|L)^2} = 
\frac{\mu_{G,L}(I_{Q \cap L}^L (\sigma \otimes \chi))}{c(G|L)^2} = 
\frac{\mu_{G,M} (\sigma \otimes \chi)}{\mu_{L,M}(\sigma \otimes \chi)}
\frac{c(L|M)^2}{c(G|M)^2} . \qedhere
\]
\end{proof}

Propositon \ref{prop:1.2} enables us to reduce the computation of $\mu$-functions to the case of cuspidal representations, which is already well-understood.

Let $\sigma \in \Irr (M)$ be cuspidal. For $\alpha \in \Phi (G,A_M)^+$, let 
$h_\alpha^\vee \in M / M^1$ be as in \cite[Appendix]{SolEnd} and \cite{FlSo}. This
element $h_\alpha^\vee$ depends on $X_\nr (M) \sigma$ and plays a role similar to
a coroot $\alpha^\vee$. If $N_{M_\alpha}(M) \neq M$, we pick an element 
$s_\alpha \in N_{M_\alpha}(M) \setminus M$.

\begin{thm}\label{thm:6.2} 
\textup{\cite[Theorem 1.6]{Sil3} and \cite[Theorem 1.2]{FlSo}} 
\enuma{
\item If $j_{M_\alpha,M}$ does not have a pole on $X_\nr (M) \sigma$, then
it equals a constant function $c_\alpha \in \R_{>0}$ on $X_\nr (M) \sigma$.
This happens whenever $N_{M_\alpha} (M) = M$ or $N_{M_\alpha} (M) \neq M$ and 
$s_\alpha$ does not stabilize $X_\nr (M) \sigma$.
\item Suppose that $j_{M_\alpha,M}$ has a pole on $X_\nr (M) \sigma$. By moving $\sigma$ 
inside $X_\nr (M) \sigma$, we can arrange that $\sigma$ is unitary, $j_{G,M}(\sigma) = 
\infty$ and $s_\alpha$ fixes $\sigma$. Then there exist $c_\alpha \in \R_{>0},
q_\alpha \in \R_{>1}, q_{\alpha *} \in \R_{\geq 1}$ such that
\[
j_{M_\alpha,M}(\sigma \otimes \chi) = c_\alpha \frac{(1 - q_\alpha \chi (h_\alpha^\vee))
(1 - q_\alpha \chi (h_\alpha^\vee)^{-1})}{(1 - \chi (h_\alpha^\vee))
(1 - \chi (h_\alpha^\vee)^{-1})} \frac{(1 + q_{\alpha *} \chi (h_\alpha^\vee))
(1 + q_{\alpha *} \chi (h_\alpha^\vee)^{-1})}{(1 + \chi (h_\alpha^\vee))
(1 + \chi (h_\alpha^\vee)^{-1})} 
\]
for all $\chi \in X_\nr (M)$.
}
\end{thm}

In Theorem \ref{thm:6.2}.b, $q_{\alpha *} = 1$ if $2 \alpha$ is not a root of $(G,A_M)$. 
Theorem \ref{thm:6.2}.a can be described by the same formula as part b, 
namely with $q_\alpha = q_{\alpha *} = 1$. \\

Consider a cuspidal Bernstein component $X_\nr (M) \sigma'$ in $\Irr (M)$. Let\\
$\Phi (G,A_M, X_\nr (M) \sigma')$ be the set of those $\alpha \in \Phi (G,A_M)$ for which 
$\mu_{M_\alpha,M}$ has a zero (or equivalently is not constant) on $X_\nr (M) \sigma'$. 
By \cite[Proposition 1.3]{Hei2}, $\Phi (G,A_M, X_\nr (M) \sigma')$ is a reduced root system
whose Weyl group embeds canonically in $N_G (M) / M$. 
The following result helps us to apply Theorem \ref{thm:6.2} simultaneously to 
several roots from $\Phi (G,A_M, X_\nr (M) \sigma')$.

\begin{lem}\label{lem:1.3}
There exists a unitary $\sigma \in X_\nr (M) \sigma'$ such that 
$\mu_{M_\alpha,M}(\sigma) = 0$ and $s_\alpha \cdot \sigma \cong \sigma$ for all 
$\alpha \in \Phi (G,A_M, X_\nr (M) \sigma')$.
\end{lem}
\begin{proof}
A parabolic subgroup $P' = M U_{P'} \subset G$ determines which roots in\\
$\Phi (G,A_M, X_\nr (M) \sigma')$ are positive and
which are simple. The simple roots are linearly independent so, as already observed in 
\cite{Hei2}, one can find a unitary $\sigma \in X_\nr (M) \sigma'$ such that 
$\mu_{M_\alpha,M}(\sigma) = 0$ for all simple $\alpha \in \Phi (G,A_M, X_\nr (M) \sigma')$.
By \cite[\S 5.4.2]{Sil0}, $s_\alpha \sigma \cong \sigma$ for all such $\alpha$. Hence
$W \big( \Phi (G,A_M, X_\nr (M) \sigma') \big)$ fixes $\sigma$ (up to isomorphism).
Given any $\beta \in \Phi (G,A_M, X_\nr (M) \sigma')$, there exists a 
\[
w \in W \big( \Phi (G,A_M, X_\nr (M) \sigma') \big) \subset N_G (M) / M
\]
such that $\beta = w (\alpha)$ for a simple root $\alpha$. Via an isomorphism
$w^{-1} \sigma \cong \sigma$, we can identify
\begin{align*}
J_{M U_{-\beta} | M U_{\beta}}(\sigma \otimes \chi) & = w \circ J_{M U_{-\alpha} | M U_\alpha} 
(w^{-1} (\sigma \otimes \chi)) \circ w^{-1} \\
& = w \circ J_{M U_{-\alpha} | M U_\alpha} (\sigma \otimes w^{-1} \chi) \circ w^{-1} .
\end{align*}
The same holds for $-\beta$, which entails that 
\begin{equation}\label{eq:1.8}
j_{M_\beta,M}(\sigma \otimes \chi) = j_{M_\alpha,M}(\sigma \otimes w^{-1} \chi) \quad
\text{for all} \quad \chi \in X_\nr (M). 
\end{equation}
By \eqref{eq:1.3} and \eqref{eq:1.8} we have
$\mu_{M_\beta,M}(\sigma) = \mu_{M_\alpha,M}(\sigma) = 0$.
\end{proof}

We are ready to state an explicit formula for Harish-Chandra's function $\mu_{G,L}$, for 
any irreducible $L$-representation. 

\begin{thm}\label{thm:1.4}
Let $\pi \in \Irr (L)$. Suppose that $(M,\sigma \otimes \chi_\pi)$ represents the cuspidal
support of $\pi$, where $\sigma$ is as in Lemma \ref{lem:1.3} and $\chi_\pi \in X_\nr (M)$. 
Then there exists $c \in \R_{>0}$, depending only on $X_\nr (M) \sigma$ and $G$, such that
\begin{multline*}
\mu_{G,L}(\pi \otimes \chi) = c
\prod_{\alpha \in \Phi (G,A_M)^+ \setminus \Phi (L,A_M)^+} 
\frac{(1 - (\chi_\pi \chi) (h_\alpha^\vee)) (1 - (\chi_\pi \chi) 
(h_\alpha^\vee)^{-1})}{(1 - q_\alpha (\chi_\pi \chi)
(h_\alpha^\vee)) (1 - q_\alpha (\chi_\pi \chi)(h_\alpha^\vee)^{-1})} \\
\cdot \frac{(1 + (\chi_\pi \chi) (h_\alpha^\vee))
(1 + (\chi_\pi \chi) (h_\alpha^\vee)^{-1})} {(1 + q_{\alpha *} (\chi_\pi \chi) (h_\alpha^\vee))
(1 + q_{\alpha *} (\chi_\pi \chi) (h_\alpha^\vee)^{-1})} 
\end{multline*}
as rational functions of $\chi \in X_\nr (L)$.
\end{thm}
\begin{proof}
By Proposition \ref{prop:1.2} and \eqref{eq:1.4} we have
\[
\mu_{G,L}(\pi \otimes \chi) = \frac{c(G|L)^2 c(L|M)^2}{c(G|M)^2}
\frac{\prod_{\alpha \in \Phi (G,A_M)^+} \mu_{M_\alpha,M}(\sigma \otimes \chi_\pi \chi)}{
\prod_{\alpha \in \Phi (l,A_M)^+} \mu_{M_\alpha,M}(\sigma \otimes \chi_\pi \chi)} .
\]
Combine that with Theorem \ref{thm:6.2} and \eqref{eq:1.3}. Lemma \ref{lem:1.3} guarantees
that $\sigma$ is in the position required in Theorem \ref{thm:6.2}.b, for any
$\alpha \in \Phi (G,A_M, X_\nr (M) \sigma')$.
\end{proof}

\begin{rem}
Consider a finite central cover $\tilde G$ of the topological group $G$. The results in this 
paragraph hold just as well for $\tilde G$. The reason is that every unipotent subgroup of 
$G$ admits a canonical lifting to $\tilde G$ \cite[\S A.1]{MWa}, so that one can reason in
$\tilde G$ with unipotent subgroups exactly like in $G$. Therefore our proofs apply also 
to $\tilde G$. Theorem \ref{thm:6.2} was already proven in that generality in \cite{FlSo}.
\end{rem}

\subsection{Normalized intertwining operators} \

Consider a cuspidal $\sigma \in \Irr (M)$ and $\alpha \in \Phi (G,A_M )$ such that
$\mu_{M_\alpha,M} (\sigma) = 0$. We define a normalized version of 
$J_{M U_{-\alpha} | M U_\alpha} (\sigma \otimes \chi)$ by
\[
J'_{M U_{-\alpha} | M U_\alpha} (\sigma \otimes \chi) =
(\chi (h_\alpha^\vee) - 1) J_{M U_{-\alpha} | M U_\alpha} (\sigma \otimes \chi)
\qquad \chi \in X_\nr (M) .
\]
According to \cite[Lemme 1.8]{Hei2}, 
\begin{equation}\label{eq:1.11}
J'_{M U_{-\alpha} | M U_\alpha} (\sigma \otimes \chi) \text{ is invertible for }
\chi \text{ in a neighborhood of 1 in } X_\nr (M) .
\end{equation}
More generally, let $Q$ and $Q'$ be parabolic subgroups of $G$ with Levi factor $M$.
We define the normalization of $J_{Q'|Q}(\sigma \otimes \chi)$ as 
\[
J'_{Q'|Q}(\sigma \otimes \chi) = J_{Q'|Q}(\sigma \otimes \chi) \prod_{\alpha \in \Phi 
(U_Q ,A_M) \cap\Phi (U_{\overline{Q'}},A_M) : \mu_{M_\alpha,M}(\sigma) = 0} 
(\chi (h_\alpha^\vee) - 1) .
\]
By reduction to \eqref{eq:1.11} one shows:

\begin{prop}\label{prop:1.6}
Suppose that $\mu_{M_\alpha,M}(\sigma) \neq \infty$ for all $\alpha \in 
\Phi (U_Q ,A_M) \cap \Phi (U_{\overline{Q'}},A_M)$. Then there exists a neighborhood
$V_1$ of 1 in $X_\nr (M)$, such that 
\[
J'_{Q'|Q}(\sigma \otimes \chi) : I_Q^G (\sigma \otimes \chi) \to
I_{Q'}^G (\sigma \otimes \chi)
\]
is an isomorphism of $G$-representations for all $\chi \in V_1$.
\end{prop}

There are also normalized intertwining operators for non-cuspidal 
representations. Let $\pi \in \Irr (L)$ be a subquotient of $I_{Q \cap L}^Q (\sigma)$
and write $P = Q L$. For another parabolic subgroup $P' \subset G$ with Levi
factor $L$ we define
\begin{equation} \label{eq:1.12}
J'_{P'|P}(\pi \otimes \chi) =  J'_{P'|P}(\pi \otimes \chi) 
\prod_{\alpha \in \Phi (U_P,A_M) \cap \Phi (U_{\overline{P'}}, A_M) : 
\mu_{M_\alpha,M}(\sigma) = 0} (\chi (h_\alpha^\vee) - 1) .
\end{equation}

\begin{thm}\label{thm:1.7}
Suppose that $\mu_{M_\alpha,M}(\sigma) \neq \infty$ for all $\alpha \in 
\Phi (U_P ,A_M) \cap \Phi (U_{\overline{P'}},A_M)$. Then there exists a neighborhood
$V'_1$ of 1 in $X_\nr (L)$, such that 
\[
J'_{P'|P}(\pi \otimes \chi) : I_P^G (\pi \otimes \chi) \to 
I_{P'}^G (\pi \otimes \chi)
\]
is an isomorphism of $G$-representations, for all $\chi \in V'_1$.
\end{thm}
\begin{proof}
The set of roots $\Phi (U_{P'},A_M) \cup \Phi (U_{Q \cap L}, A_M)$ is a positive
system in $\Phi (G,A_M)$. That gives a parabolic subgroup $Q'$ of $G$ with Levi
factor $M$, such that $P' = Q' L$ and $Q' \cap L = Q \cap L$. Now
\[
\Phi (U_P ,A_M) \cap \Phi (U_{\overline{P'}},A_M) =
\Phi (U_Q ,A_M) \cap \Phi (U_{\overline{Q'}},A_M) ,
\]
which means that the normalization factors $\prod_\alpha (\chi (h_\alpha^\vee) - 1)$
are the same for $J_{Q'|Q}(\sigma \otimes \chi)$ and $J_{P'|P} \big( I_{Q \cap L}^L
(\sigma \otimes \chi) \big)$. With an argument like in \eqref{eq:1.6}--\eqref{eq:1.7} 
we obtain
\[
J'_{Q'|Q}(\sigma \otimes \chi) = J'_{QL'|QL} \big( I_{Q \cap L}^L (\sigma \otimes 
\chi) \big) : I_Q^G (\sigma \otimes \chi) \to I_{Q'}^G (\sigma \otimes \chi) .
\]
By Proposition \ref{prop:1.6}, $J'_{Q'|Q}(\sigma \otimes \chi)$ is an isomorphism
for $\chi \in V_1 \subset X_\nr (M)$. Hence 
\[
J'_{P'|P} \big( I_{Q \cap L}^L (\sigma) \otimes \chi \big) = 
J'_{Q'L|QL} \big( I_{Q \cap L}^L (\sigma \otimes \chi) \big)
\]
is an isomorphism for $\chi \in V'_1 := \{ \chi \in X_\nr (L) : \chi |_M \in V_1 \}$.
Pick subrepresentations $\pi_1,\pi_2$ of $I_{Q \cap L}^L (\sigma)$ such that 
$\pi = \pi_1 / \pi_2$. By the above
\[
J'_{P'|P} (\pi_i \otimes \chi) : I_P^G (\pi_i \otimes \chi) \to 
I_{P'}^G (\pi_i \otimes \chi) \qquad i = 1,2, \chi \in V'_1 
\]
are isomorphisms. The map for $i=1$ extends the map for $i=2$, and passing to the
quotient $\pi = \pi_1 / \pi_2$ we find that
\[
J'_{P'|P}(\pi \otimes \chi) : I_P^G (\pi \otimes \chi) \to I_{P'}^G (\pi \otimes
\chi) \qquad \chi \in V'_1 
\]
is also an isomorphism.
\end{proof}

\subsection{Residual points of the $\mu$-functions} \

From \eqref{eq:1.3} and \eqref{eq:1.4} we see that $\mu_{G,M}$ can only have 
a pole at $\sigma \otimes \chi$ if at least one of the functions $j_{M_\alpha, M}$ 
has a zero at $\sigma \otimes \chi$. Then 
\[
J_{M U_\alpha | M U_{-\alpha}} (\sigma \otimes \chi) 
J_{M U_{-\alpha} | M U_\alpha} (\sigma \otimes \chi) = 0
\]
and both factors are nonzero, so both $J_{M U_\alpha | M U_{-\alpha}} (\sigma 
\otimes \chi)$ and $J_{M U_{-\alpha} | M U_\alpha} (\sigma \otimes \chi)$ are not
injective. We write 
\[
r_M = \dim_F Z(M) - \dim_F Z(G) = \dim_F A_M - \dim_F A_G .
\]
\begin{defn}
A representation $\sigma \otimes \chi$ at which $\mu_{G,M}$ has
a pole of order $r_M$ is called a residual point of $\mu_{G,M}$.     
\end{defn}
The function $\mu_{G,M}$ has no poles of order $> r_M$ \cite[Corollaire 8.6]{Hei1}.

\begin{thm}\label{thm:6.7}
\textup{\cite[Th\'eor\`eme 8.6 and Corollaire 8.7]{Hei1}} \\
The representation $I_Q^G (\sigma \otimes \chi)$ has an essentially square-integrable
subquotient if and only if $\mu_{G,M}$ has a pole of order $r_M$ at $\sigma \otimes \chi$. 
Such a subquotient is square-integrable modulo centre if and only if
$|cc(\sigma \otimes \chi) |_{Z(G)}  = 1$.
\end{thm}

Theorem \ref{thm:6.7} says that the cuspidal supports of essentially square-inte\-gra\-ble
representations are precisely the residual points of the Harish-Chandra $\mu$-functions.
We note that $X_\nr (M) \sigma$ does not always contain residual points. A necessary
condition is that $\Phi (G,M, X_\nr (M) \sigma)$ has rank $r_M$ 
\cite[Proposition A.3.(1)]{Opd}.

\begin{lem}\label{lem:6.10}
\enuma{
\item There are only finitely many $X_\nr (G)$-orbits of re\-si\-dual points for 
$\mu_{G,M}$ in $X_\nr (M) \sigma$.
\item Suppose that $Z(G)$ is compact and that $\sigma$ is as in Lemma \ref{lem:1.3}.
Then every residual point $\sigma \otimes \chi$ satisfies a collection of equations
\[
\chi (h_\alpha^\vee) = q, \quad \text{where} \quad 
q \in \{\pm q_\alpha, \pm q_\alpha^{-1}, \pm q_{\alpha*}, \pm q_{\alpha*}^{-1}, \pm 1\}
\]
in the notation of Theorem \ref{thm:6.2} for $M_\alpha \supset M$, and $\alpha$ runs 
through a subset of $\Phi (G,M,X_\nr (M)\sigma)^+$ whose $\Q$-span has dimension $r_M$.

Conversely, this collection of equations determines $\chi$ up to a finite subgroup
of $X_\nr (M)$. 
\item In the setting of part (b), $|cc(\sigma \otimes \chi)|$ is determined by a
collection of equations 
\[
\big| cc (\sigma' \otimes \chi)(h_\alpha^{\vee N_\alpha}) \big| = 
\chi \big( h_\alpha^{\vee N_\alpha} \big) \in 
\{q_\alpha^{N_\alpha}, q_\alpha^{-N_\alpha}, q_{\alpha *}^{N_\alpha}, 
q_{\alpha *}^{-N_\alpha}, 1\}, 
\]
with the same $\alpha$ as in (b) and some $N_\alpha \in 2 \Z_{>0}$.
}    
\end{lem}
\begin{proof}
(a) This is a special case of \cite[Corollary A.2]{Opd}.\\
(b) As we saw above, $\Phi (G,M,X_\nr (M) \sigma)$ must have rank $r_M = 
\dim_\C X_\nr (M) \sigma$. By part (a) and the compactness of $Z(G)$, $\mu_{G,M}$ has only 
finitely many residual points in $X_\nr (M) \sigma$. By \eqref{eq:1.3} and 
\cite[Theorem A.7]{Opd}, there exist $r_M$ linearly independent roots 
$\alpha \in \Phi (G,M,X_\nr (M)\sigma)^+$ such that 
\[
\mu_{M_\alpha,M}(\sigma) = 0 \quad \text{and} \quad 
j_{M_\alpha,M}(\sigma \otimes \chi) = \infty. 
\]
By Theorem \ref{thm:6.2} and \eqref{eq:1.3}, these equations imply that 
$\chi (h_\alpha^\vee)$ or $\chi (h_\alpha^\vee)^{-1}$ lies in $\{\pm q_\alpha,\pm q_{\alpha*} \}$. 
There may be further $\alpha \in \Phi (G,M,X_\nr (M)\sigma)^+$ with  
$\chi (h_\alpha^\vee)$  or $\chi (h_\alpha^\vee)^{-1}$ in $\{\pm q_\alpha,\pm q_{\alpha*},\pm 1\}$, 
in the notations from Theorem \ref{thm:6.2} for $j_{M_\alpha,M}$. 
We include those as equations for $\sigma \otimes \chi$.

The elements $h_\alpha^\vee$ for the $\alpha$ as above span a finite index sublattice of 
$M / M^1$. Therefore the values $\chi (h_\alpha^\vee)$ determine $\chi$ up to a finite
subgroup of $X_\nr (M)$.\\
(c) Recall from Lemma \ref{lem:1.3} that $\sigma$ is unitary. 
Since $Z(M) M^1$ has finite index in $M$, we can find $N_\alpha \in 2 \Z_{>0}$
such that $h_\alpha^{\vee N_\alpha} \in Z(M) M^1 / M^1$. For any representative
$h \in Z(M)$ of $h_\alpha^{\vee N_\alpha}$, part (b) shows that
\begin{equation}\label{eq:6.26}
|cc (\sigma' \otimes \chi)(h)| = |\chi (h)| = \big| \chi \big( h_\alpha^{\vee N_\alpha} 
\big) \big| = \chi \big( h_\alpha^{\vee N_\alpha} \big) \in \{ q^{\pm N_\alpha}, 
q'^{\pm N_\alpha}, 1 \} .
\end{equation}
We may define this number to be $\big| cc (\sigma' \otimes \chi)(h_\alpha^{\vee N_\alpha}) 
\big|$. The numbers \eqref{eq:6.26}, for all $\alpha$ as in part (b), determine
$|cc(\sigma' \otimes \chi)| \in \Hom (Z(M),\R_{>0})$ on a finite index sublattice of
$Z(M) / Z(M)^1$. Since all $n$-th roots are unique in $\R_{>0}$, that determines
$|cc(\sigma' \otimes \chi)|$ completely.
\end{proof}

\subsection{Analytic R-groups} \
\label{par:1.5}

Let $P = L U_P$ be a parabolic subgroup of $G$. The group $N_G (L)$ acts naturally on 
$\Irr (L)$, by $(n \cdot \pi)(l) = \pi (n^{-1} l n)$. This descends to an action of 
\[
W_L := N_G (L) / L
\] 
on $\Irr (L)$, which sends $X_\nr (L)$ to $X_\nr (L)$.
Let $W_{L,\pi}$ be the stabilizer of $\pi \in \Irr (L)$ in $W_L$.

Let $\delta \in \Irr (L)$ be essentially square-integrable. 
Consider the set of reduced roots $\alpha$ of $(G,A_L)$ such that Harish-Chandra's 
function $\mu_{L_\alpha,L}$ has a zero at $\delta$. These roots form a finite integral 
root system \cite{Sil5}, say $\Phi (G,A_L,\delta)$. The group
$W_{L, \delta}$ acts on $\Phi (G,A_L,\delta)$ and contains the Weyl group
$W(\Phi (G,A_L,\delta) )$ as a normal subgroup. Let $\Phi (G,A_L,\delta)^+$ be the
positive system of roots appearing in the Lie algebra of $P$. The analytic R-group
$R_\delta$ is defined as the stabilizer of $\Phi (G,A_L,\delta )^+$ in
$W_{L,\delta}$. Since $W(\Phi (G,A_L,\delta) )$ acts simply transitively
on the collection of positive systems in $\Phi (G,A_L,\delta)$, we have a decomposition
\begin{equation}\label{eq:5.5}
W_{L,\delta} = W(\Phi (G,A_L,\delta) ) \rtimes R_\delta .
\end{equation}
This is a generalization of the R-groups from \cite{Ar} because we allow
non-tempered representations $\delta$, but apart from that it is the same definition.

Every $w \in W_{L,\delta}$ gives rise to an intertwining operator
$J_\delta (w) \in \mr{Aut}_G (I_P^G (\delta))$ \cite[Lemma 1.3]{ABPS}, unique up to 
scalars. It arises from the normalized intertwining operators \eqref{eq:1.12} by a
further normalization (to make it unitary if $\delta$ is tempered) and translation 
along $w$. By results of Knapp--Stein \cite{Sil5}, and by \cite[Lemma 1.5]{ABPS} in the 
non-tempered cases, $J_\delta (w)$ is a scalar multiple of the identity if and only if
$w \in W(\Phi (G,A_L,\delta) )$. Therefore it suffices to consider the intertwining
operators $J_\delta (r)$ with $r \in R_\delta $. These operators span a twisted group 
algebra $\C [R_\delta, \natural_\delta ]$, for some 2-cocycle $R_\delta \times 
R_\delta \to \C^\times$. In other words, $J_\delta$ yields a projective representation
of $R_\delta$ on $I_P^G (\delta )$. By \cite[Theorem 1.6]{ABPS}
there is a decomposition of $\C [R_\delta,\natural_\delta] \times \C G$-modules
\begin{equation}\label{eq:7.9}
\begin{aligned}
& I_P^G (\delta ) = \bigoplus\nolimits_{\kappa \in \Irr \, \C [R_\delta , \natural_\delta]} 
\kappa \otimes I_P^G (\delta )_\kappa , \\
& I_P^G (\delta )_\kappa = \Hom_{\C [R_\delta,\natural_\delta]} (\kappa, I_P^G (\delta)).
\end{aligned}
\end{equation}
If $\delta$ is square-integrable modulo centre (so in particular tempered), then
$\C [R_\delta, \natural_\delta ]$ equals $\End_G (I_P^G (\delta))$ and all the 
representations $I_P^G (\delta )_\kappa$ are irreducible \cite{Sil5}.

\section{Quasi-standard modules} \label{sec:2}

Le $L \subset G$ be a Levi subgroup and let $S \subset L$ be a maximal $F$-split
torus. Then $S$ is the maximal split central torus in the Levi subgroup $Z_G (S)$,
and $\Phi (G,S)$ is the set of reduced roots of $(G,S)$. This is a reduced integral 
root system in $X^* (S)$, and there is a coroot system $\Phi (G,S)^\vee$ in
$X_* (S)$. We recall from \cite[\S V.3.13]{Ren} that there are canonical 
decompositions
\begin{equation}\label{eq:2.5}
\begin{array}{ccc}
X^* (S) \otimes_\Z \R & = & X^* (S \cap L_\der) \otimes_\Z \R \; \oplus \;
X^* (A_L) \otimes_\Z \R ,\\
X_* (S) \otimes_\Z \R & = & X_* (S \cap L_\der) \otimes_\Z \R \; \oplus \;
X_* (A_L) \otimes_\Z \R .
\end{array}
\end{equation}
Every $\alpha \in \Phi (G,A_L) \subset X^* (A_L)$ can be extended to an
element $\alpha_S \in \Phi (G,S)$, usually in several ways. We define 
$\alpha^\vee$ as the projection of $\alpha_S^\vee$ to $X_* (A_L) \otimes_\Z \R$
via \eqref{eq:2.5}. This does not depend on the choice of $\alpha_S$ because
$X^* (S \cap L_\der)$ is orthogonal to $X_* (A_L)$. It does not depend on $S$
either, because all maximal $F$-split tori of $L$ are conjugate.

Let $P = L U_P$ be a parabolic subgroup of $G$ with Levi factor $L$ and let $\nu 
\in \Hom (L,\R_{>0})$, so $\log \nu \in \Hom (L,\R) \cong X^* (A_L) \otimes_\Z \R$.
We say that $\nu$ is strictly positive with respect to $P$ if 
$\langle \alpha^\vee, \log \nu \rangle > 0$ for all $\alpha \in \Phi (P,A_L) 
= \Phi (U_P, A_L)$. This condition is equivalent to:
\begin{equation}\label{eq:2.7}
\langle \alpha_S^\vee ,\log \nu \rangle > 0 \quad \forall \alpha_S \in 
\Phi (G,S) \text{ with } \alpha_S |_{A_L} \in \Phi (U_P,A_L) .
\end{equation}
Let $\delta \in \Irr (L)$ be essentially square-integrable and let 
$cc(\delta) : Z(L) \to \C^\times$ be its central character. We note that 
$|cc(\delta)|$ is determined
by its restriction to $A_L$, because $L_\der A_L$ is cocompact in~$L$.

\begin{defn}
We call $(P,L,\delta)$ an induction datum for $G$. We say that $(P,L,\delta)$
is positive if $\langle \alpha^\vee , \log |cc (\delta)| \rangle \geq 0$ for
all roots $\alpha$ of $(P,A_L)$. If $\tilde \delta \cong \delta$ then
$(P,L, \tilde \delta)$ is considered as equivalent to $(P,L,\delta)$.
\end{defn}

Recall the R-group $R_\delta$, the twisted group algebra $\C [R_\delta,\natural_\delta]$
and the decomposition of $I_P^G (\delta)$ from \eqref{eq:7.9}.

\begin{defn}\label{def:7.1}
Let $(P,L,\delta)$ be an induction datum  and let $\kappa \in \Irr \,
\C [R_\delta, \natural_\delta]$.
A $\C G$-module is called quasi-standard if it has the form
$I_P^G (\delta )_\kappa$ as in \eqref{eq:7.9}.
\end{defn}
This terminology is motivated by the following result.

\begin{thm}\label{thm:7.2}
\textup{\cite[\S 2.4]{Sol1} and \cite[\S 1]{ABPS}}\\
Let $(P,L,\delta)$ be an induction datum and let $\kappa \in \Irr \, \C [R_\delta,
\natural_\delta]$.
\enuma{
\item If $(P,L,\delta)$ is positive, then $\C [R_\delta, \natural_\delta] = 
\End_G (I_P^G (\delta))$ and $I_P^G (\delta)_\kappa$ is a
standard $\C G$-module.
\item Every standard $\C G$-module arises as in part (a), from $(P,L,\delta,
\kappa)$ which are unique up to $G$-conjugation.
\item Let $\pi \in \Irr (G)$. There exists a positive induction datum,
unique up to $G$-conjugation, such that $\pi$ is a quotient of $I_P^G (\delta)$.
}
\end{thm}

Let $L_\delta \supset L$ be the largest Levi subgroup of $G$ such that
$|cc(\delta)| = 1$ on $Z(L) \cap L_{\delta,\der}$. By \cite[Theorem 1.6]{ABPS},
$I_{L_\delta \cap P}^{L_\delta} (\delta)$ is completely reducible and
decomposes as
\begin{equation}\label{eq:7.2}
I_{L_\delta \cap P}^{L_\delta} (\delta) =
\bigoplus\nolimits_{\kappa \in \Irr \, \C [R_\delta, \natural_\delta]}
\kappa \otimes I_{L_\delta \cap P}^{L_\delta} (\delta )_\kappa .
\end{equation}
Moreover, each $I_{L_\delta \cap P}^{L_\delta} (\delta )_\kappa$ can be
written as $\tau \otimes \nu$ where $\tau$ is an irreducible tempered
$L_\delta$-representation and $\nu \in \Hom (L_\delta, \R_{>0})$. By
the definition of $L_\delta$, $\nu$ does not extend to a character of
any Levi subgroup of $G$ strictly containing $L_\delta$. We note that,
by the transitivity of parabolic induction
\begin{equation}\label{eq:7.11}
I_P^G  (\delta )_\kappa \cong I_{L_\delta P}^G \big( I_{L_\delta \cap P}^{
L_\delta} (\delta )_\kappa \big) \cong I_{L_\delta P}^G (\tau \otimes \nu) .
\end{equation}
Therefore one can characterize quasi-standard $\C G$-modules as 
representations of the form $I_Q^G (\tau \otimes \nu)$, where $Q = M U_Q$ is
a parabolic subgroup of $G$, $\tau \in \Irr (M)$ is tempered and 
$\nu \in \Hom (M, \R_{>0})$ does not extend to a character of any strictly
larger Levi subgroup of $G$. The difference with standard $\C G$-modules
is that $\nu$ need not be positive with respect to $Q$. 

Let $Q_\nu$ be the parabolic subgroup of $G$ with Levi factor $M$ and
unipotent radical generated by root subgroups $\alpha$ with 
$\langle \alpha^\vee, \log \nu \rangle > 0$. Then $I_{Q_\nu}^G (\tau \otimes \nu)$ 
is standard. In the same way every induction datum $(P,L,\delta)$ 
can be made positive by changing only $P$.

We say that two induction data $(P,L,\delta)$ and $(P',L',\omega)$ are
$G$-associate if there exists a $g \in G$ such that $g L g^{-1} = L'$
and $g \cdot \delta \cong \omega$. It is known from \cite[Lemma 2.13]{Sol1}
that every induction datum is $G$-associate to a positive induction datum,
unique up to equivalence.

For two associate induction data as above we have
\begin{equation}\label{eq:2.6}
I_P^G (\delta) \cong I_{g P g^{-1}}^G (g \cdot \delta) \cong
I_{g P g^{-1}}^G (\omega) .
\end{equation}
By \eqref{eq:2.6} and \cite[Lemma 1.1]{ABPS}
\begin{equation}\label{eq:7.1}
I_P^G (\delta) \quad \text{and} \quad I_{P'}^G (\omega) \quad
\text{have the same Jordan--H\"older content.}
\end{equation}
We proceed to make this statement more precise. The group $g L_\delta g^{-1}
= L'_\omega$ has the same properties as $L_\delta$, only for
$(P',L',\omega)$. By \cite[Lemma 1.1]{ABPS} the $L'_\omega$-representations
$g \cdot I_{L_\delta \cap P}^{L_\delta} (\delta) \cong I_{L'_\omega \cap
g P g^{-1}}^{L'_\omega} (\omega)$ and $I_{L'_\omega \cap P'}^{L'_\omega}
(\omega)$ have the same Jordan--H\"older content. Since they are both
completely reducible, we conclude that
\begin{equation}\label{eq:7.3}
g \cdot I_{L_\delta \cap P}^{L_\delta} (\delta) \cong
I_{L'_\omega \cap P'}^{L'_\omega} (\omega) .
\end{equation}
Conjugation by $g$ induces a group isomorphism $R_\delta \cong R_\omega$
and a bijection
\[
\Irr \, \C [R_\delta,\natural_\delta] \to \Irr \, \C [R_\omega,
\natural_\omega] : \kappa \mapsto \kappa' .
\]
Together with \eqref{eq:7.2} and \eqref{eq:7.3} that implies
\begin{equation}\label{eq:7.4}
g (\kappa \otimes I_{L_\delta \cap P}^{L_\delta} (\delta)_\kappa) \cong
\kappa' \otimes I_{L'_\omega \cap P'}^{L'_\omega} (\omega )_{\kappa'} .
\end{equation}

\begin{lem}\label{lem:7.7}
In the setting of \eqref{eq:7.4}, the representations $I_P^G (\delta)_\kappa$
and $I_{P'}^G (\omega)_{\kappa'}$
have the same Jordan--H\"older content. Moreover, there exists a nonzero
$G$-intertwining operator $I_P^G (\delta)_\kappa \to I_{P'}^G (\omega)_{\kappa'}$.
\end{lem}
\begin{proof}
We abbreviate $\tau' = I_{L_\delta \cap P}^{L_\delta} (\delta)_\kappa$, so
that $I_P^G (\delta)_\kappa \cong I_{L_\delta P}^G (\tau')$.
By \eqref{eq:7.4} there are isomorphisms
\begin{equation}\label{eq:7.8}
I_{P'}^G (\omega)_{\kappa'} \cong I_{L'_\omega P}^G \big( I_{L'_\omega \cap P'
}^{L'_\omega} (\omega )_{\kappa'} \big) \cong I_{L'_\omega P'}^G (g \cdot
\tau') \cong I_{L_\delta g^{-1} P' g}^G (\tau') .
\end{equation}
By \cite[Lemma 1.1]{ABPS} $I_{L_\delta P}^G (\tau')$ and $I_{L_\delta g^{-1}
P' g}^G (\tau')$ have the same Jordan--H\"older content.
We recall Harish-Chandra's intertwining operators
\begin{equation}\label{eq:7.7}
J_{L_\delta g^{-1} P' g | L_\delta P}(\tau' \otimes \chi) : 
I_{L_\delta P}^G (\tau' \otimes \chi) \to I_{L_\delta g^{-1} P g}^G 
(\tau' \otimes \chi) \quad \chi \in X_\nr (L_\delta) .
\end{equation}
from \eqref{eq:6.1}. As we saw in \eqref{eq:1.9}--\eqref{eq:1.10},
$J_{L_\delta g^{-1} P' g | L_\delta P}(\tau' \otimes \chi)$ is a composition of 
intertwining operators from a corank one setting. Let 
$J_{P_2 | P_1}(\tau' \otimes \chi)$ be a such a simpler intertwining operator 
and let $L_{12}$ be the derived group of the group generated by $P_1 \cup P_2$.
By Theorem \ref{thm:6.2} or \cite[p. 283]{Wal}, $J_{P_2 | P_1}(\tau' \otimes \chi)$ 
can only have a pole at $\chi = 1$ if the nontrivial element $s_\alpha$ of the
associated Weyl group (a subgroup of $N_G (L_\delta) / L_\delta$ of order
at most two) stabilizes $\tau'$. That would imply that $|cc(\tau')| =
|cc(\delta)|_{Z(L_\delta)}$ is trivial on $Z(L_{12}) \supsetneq
Z(L_{\delta,\der})$. But that would contradict the construction of
$L_\delta$, so $J_{P_2 | P_1}(\tau' \otimes \chi)$ is regular at $\chi = 1$.

Hence $J_{L_\delta g^{-1} P' g | L_\delta P}(\tau' \otimes \chi)$ is regular
at $\chi = 1$ and \eqref{eq:7.7} is well-defined. Then \cite[p. 283]{Wal}
shows that \eqref{eq:7.7} is nonzero. Finally, we compose \eqref{eq:7.7}
with the isomorphism \eqref{eq:7.8} (from right to left).
\end{proof}

\subsection{A rank one case} \

We work out quasi-standard modules in a relevant simple case, which is known but for 
which we could not find a good reference.

\begin{prop}\label{prop:2.1}
Let $\sigma \in \Irr (M)$ be unitary supercuspidal. Let 
$\nu \in \Hom (M, \R_{>0})$ and $\alpha \in \Phi (G,A_M)^+$.
\enuma{
\item If $\mu_{M_\alpha,M}(\sigma \otimes \nu) \neq \infty$, then 
$I_{M U_\alpha}^{M_\alpha} (\sigma \otimes \nu)$ is completely reducible
and has no essentially square-integrable subquotients. It has length
two if and only if $s_\alpha (\sigma \otimes \nu) \cong \sigma \otimes \nu$ and 
$\mu_{M_\alpha,M}(\sigma \otimes \nu) \neq 0$. Otherwise 
$I_{M U_\alpha}^{M_\alpha} (\sigma \otimes \nu)$ is irreducible.
\item If $\mu_{M_\alpha,M}(\sigma \otimes \nu) = \infty$, then 
$I_{M U_\alpha}^{M_\alpha} (\sigma \otimes \nu)$ has length two and is
indecomposable. If $\langle \alpha^\vee, \log \nu \rangle > 0$, then
the irreducible quotient of $I_{M U_\alpha}^{M_\alpha} (\sigma \otimes \nu)$ 
is not tempered and the irreducible subrepresentation of 
$I_{M U_\alpha}^{M_\alpha} (\sigma \otimes \nu)$ is essentially 
square-integrable. If $\langle \alpha^\vee, \log \nu \rangle < 0$, then
these properties of the quotient and the subrepresentation are exchanged.
}
\end{prop}
\begin{proof}
It is known from \cite[Th\'eor\`eme VI.5.4]{Ren} that $I_{M U_\alpha}^{M_\alpha} 
(\sigma \otimes \nu)$ has length at most two. We recall from Theorem \ref{thm:6.7}
that $I_{M U_\alpha}^{M_\alpha} (\sigma \otimes \nu)$ has an essentially
square-integrable subquotient if and only if $\mu_{M_\alpha,M}(\sigma \otimes
\nu) = \infty$.

\textbf{Case I: $\langle \alpha^\vee, \log \nu \rangle = 0$.} 
Then $M_{\sigma \otimes \nu} = M_\alpha$ and, as we saw in \eqref{eq:7.2},
$I_{M U_\alpha}^{M_\alpha} (\sigma \otimes \nu)$ is completely reducible. 
Theorem \ref{thm:6.2} shows that $\mu_{M_\alpha,M}(\sigma \otimes \nu) \neq \infty$.
Moreover \eqref{eq:7.2} shows that $I_{M U_\alpha}^{M_\alpha} (\sigma \otimes \nu)$
is irreducible whenever $R_{\sigma \otimes \nu}$ is trivial. If $R_{\sigma 
\otimes \nu}$ is nontrivial, then its only nontrivial element is $s_\alpha$,
and $\mu_{M_\alpha,M}(\sigma \otimes \nu) \neq 0$ by the definition of  
$R_{\sigma \otimes \nu}$.

\textbf{Case II: $\langle \alpha^\vee, \log \nu \rangle > 0$.}
If $N_{M_\alpha}(M) / M$ has a nontrivial element, then that does not fix
$\nu$, so in any case $W_{M_\alpha, \sigma \otimes \nu} = \{e\}$. By 
Theorem \ref{thm:7.2}.a, $I_{M U_\alpha}^{M_\alpha} (\sigma \otimes \nu)$
is a standard module, so by the Langlands classification it has a unique
irreducible quotient. As it has length at most two, it also has a unique
irreducible subrepresentation.

Suppose that $\mu_{M_\alpha,M}(\sigma \otimes \nu) = \infty$. Then
\begin{equation}\label{eq:2.1}
J_{M U_\alpha | M U_{-\alpha}}(\sigma \otimes \nu) \circ 
J_{M U_{-\alpha} | M U_\alpha}(\sigma \otimes \nu) = 
j_{M_\alpha,M}(\sigma \otimes \nu) \, \mr{id} = 0. 
\end{equation}
Both $J$-operators in \eqref{eq:2.1} are nonzero, so both are not injective.
It follows that $I_{M U_\alpha}^{M_\alpha} (\sigma \otimes \nu)$ is reducible.
By the uniqueness in the Langlands classification, its irreducible quotient
$\mc L (M U_\alpha, \sigma \otimes \nu)$ is not tempered. More precisely,
$\mc L (M U_\alpha, \sigma \otimes \nu)$ is not a tensor product of a tempered
representation and a character of $M_\alpha$, because in  that case its standard 
module would be $\mc L (M U_\alpha, \sigma \otimes \nu)$ itself. 
This also entails that the essentially square-integrable subquotient of
$I_{M U_\alpha}^{M_\alpha} (\sigma \otimes \nu)$ must be its irreducible
subrepresentation.

Suppose next that $\mu_{M_\alpha,M}(\sigma \otimes \nu) \neq \infty$,
or equivalently $j_{M_\alpha,M}(\sigma \otimes \nu) \neq 0$. From
Theorem \ref{thm:6.2} we see that $j_{M_\alpha,M}(\sigma \otimes \nu) \neq \infty$, 
so both $J_{M U_\alpha | M U_{-\alpha}}(\sigma \otimes \nu)$ and 
$J_{M U_{-\alpha} | M U_\alpha}(\sigma \otimes \nu)$ are invertible. 
By construction $\mc L (M U_\alpha, \sigma \otimes \nu)$ is the image
of $J_{M U_{-\alpha} | M U_\alpha}(\sigma \otimes \nu)$, see 
\cite[Lemme VII.4.1]{Ren}. Hence
\[
\mc L (M U_\alpha, \sigma \otimes \nu) = I_{M U_{-\alpha}}(\sigma \otimes \nu)
\cong  I_{M U_\alpha}(\sigma \otimes \nu ) .
\]
\textbf{Case III: $\langle \alpha^\vee, \log \nu \rangle < 0$.}
By \cite[(IV.2.1.2)]{Ren}, the smooth contragredient representation 
$I_{M U_\alpha}^{M_\alpha} (\sigma \otimes \nu)^\vee$ is isomorphic to
\[
I_{M U_\alpha}^{M_\alpha} ((\sigma \otimes \nu)^\vee) \cong
I_{M U_\alpha}^{M_\alpha} (\sigma^\vee \otimes \nu^{-1}) .  
\]
The representation $\sigma^\vee$ is again unitary supercuspidal, so this
brings us back to case II. According to \cite[Lemme V.2.1]{Wal}, which
is proven in \cite[Theorem 3.5]{FlSo}, 
\[
\mu_{M_\alpha,M}(\sigma^\vee \otimes \nu^{-1}) = 
\mu_{M_\alpha,M}(\sigma \otimes \nu) .
\]
When $\mu_{M_\alpha,M}(\sigma \otimes \nu) \neq \infty$, we know from case II
that $I_{M U_\alpha}^{M_\alpha} (\sigma^\vee \otimes \nu^{-1})$ is irreducible
but not essentially square-integrable. Then its contragredient 
$I_{M U_\alpha}^{M_\alpha} (\sigma \otimes \nu)^\vee$ has the same two
properties.

When $\mu_{M_\alpha,M}(\sigma \otimes \nu) = \infty$, we know from case II 
that $I_{M U_\alpha}^{M_\alpha} (\sigma^\vee \otimes \nu^{-1})$ has an
essentially square-integrable subrepresentation (say $\delta$) and a non-tempered 
irreducible quotient $\mc L (M U_\alpha , \sigma^\vee \otimes \nu^{-1})$. Then 
its contragredient $I_{M U_\alpha}^{M_\alpha} (\sigma \otimes \nu)^\vee$ has
the essentially square-integrable representation $\delta^\vee$ as quotient and
the non-tempered representation $\mc L (M U_\alpha, \sigma^\vee \otimes 
\nu^{-1})^\vee$ as subrepresentation.
\end{proof}

\subsection{An alternative characterization of standard modules} \
\label{par:2.2}

We characterize standard modules as quasi-standard modules with some 
extra properties. In this way one can avoid the use of temperedness or
positivity of characters in the description of standard modules.

We need some information about the irreducible constituents of a standard
module which are not quotients. All these are larger than the irreducible
quotient, a claim that we will quantify with an invariant from \cite{Sol1}.

We fix a maximal $F$-split torus $S$ in $G$ and a $W(G,S)$-invariant
inner product on $X^* (S) \otimes_\Z \R$. We may assume that all Levi
subgroups in our constructions are standard, in the sense that
they contain $Z_G (S)$. Alternatively we can pass to another maximal
split torus $S'$, and then the inner product transfers canonically
to $X^* (S') \otimes_\Z \R$ by its $W(G,S)$-invariance.

As before, let $\delta \in \Irr (L)$ be essentially square-integrable.
Let $(\tilde L,\sigma)$ be a representative of the cuspidal support Sc$(\delta)$
and consider cc$(\sigma) : Z(\tilde L) \to \C^\times$. As $Z(\tilde L) \tilde L^1$ 
is cocompact in $\tilde L$,
\[
\log | cc(\sigma) | : Z(\tilde L) \to \R
\]
extends uniquely to a group homomorphism from $\tilde L$ to $\R$. Then
$\log |cc (\sigma)| : \tilde L \to \R$ determines an element of
\[
\Hom (S, \R) \cong X^* (S) \otimes_\Z \R.
\]
As $\tilde L^1 S$ is cocompact in $\tilde L$, that element still determines 
$\log | cc(\sigma) |$. In these terms, the restriction of $\log | cc(\sigma) |$ to 
$\tilde L \cap L_\der$ can be described by restriction from $S$ to $S \cap L_\der$, 
so by an element of $X^* (S \cap L_\der) \otimes_\Z \R$. 
The canonical decomposition \eqref{eq:2.5} provides 
$X^* (S \cap L_\der) \otimes_\Z \R$ with a $W(L,S)$-invariant inner product.

Let $I_P^G (\delta)_\kappa$ be a quasi-standard summand of $I_P^G (\delta)$.
We define
\begin{equation}\label{eq:7.5}
\mc N (I_P^G (\delta)_\kappa) = \mc N (I_P^G (\delta)) = \mc N (\delta) =
\| \log | cc(\sigma) |_{\tilde L \cap L_\der} \| ,
\end{equation}
where the norm comes from the inner product on $X^* (S \cap L_\der) \otimes_\Z \R$.
The $W(L,S)$-invariance of this inner product implies that \eqref{eq:7.5} does 
not depend on the choice of a representative of the cuspidal support of $\delta$.
The invariant $\mc N$ measures the distance from $\delta |_{L_\der}$ to the
parabolic induction of a unitary cuspidal $\tilde L$-representation.

Clearly $|cc (\sigma) |_{\tilde L \cap L_\der}$ depends only on $\delta |_{L_\der}$.
Therefore $\mc N (I_P^G (\delta)_\kappa)$ depends only on $\delta |_{L_\der}$,
which is a direct sum of finitely many square-integrable representations $\delta_1$
\cite[Lemma 2.1 and Proposition 2.7]{Tad}. Then Sc$(\delta_1)$ can be represented by 
a subrepresentation of $\sigma |_{\tilde L \cap L_\der}$. Therefore
$\mc N (I_P^G (\delta)_\kappa)$ can be computed as
\begin{equation}\label{eq:7.6}
\mc N (I_P^G (\delta)_\kappa) = \mc N (I_P^G (\delta)) =
\| \log | cc(\mathrm{Sc}(\delta_1)) | \, \| = \mc N (\delta_1) .
\end{equation}
We note that $(P,L,\delta) \mapsto \mc N (I_P^G (\delta))$ is constant
on $G$-conjugacy classes of induction data, by the $W(G,S)$-invariance of
the inner product on $X^* (S) \otimes_\Z \R$. It is even constant on
$G$-association classes of induction data, because $P$ is inessential
in \eqref{eq:7.5}. This enables us to define, for any standard $\C G$-module
$\pi_{st}$ with irreducible quotient~$\pi$:
\[
\mc N (\pi) := \mc N (\pi_{st}).
\]

\begin{lem}\label{lem:7.3}
\textup{\cite[Lemma 2.12]{Sol1}} \\
Let $I_Q^G (\tau \otimes \nu)$ be a standard $\C G$-module with an
irreducible constituent $\pi$ different from $\mc L (Q,\tau \otimes \nu)$.
Then $\mc N (\pi) > \mc N (I_Q^G (\tau \otimes \nu))$.
\end{lem}

Easier, earlier versions of Lemma \ref{lem:7.3} have been used
to prove that the standard $\C G$-modules form a $\Z$-basis of the
Grothendieck group finite length $G$-representations. We can also use 
Lemma \ref{lem:7.3} to improve on Theorem \ref{thm:7.2}.c.

\begin{lem}\label{lem:7.4}
Suppose that a standard module $\pi_{st}$ with quotient $\pi$ is a direct
summand of $I_P^G (\delta)$, for a positive induction datum $(P,L,\delta)$.
Then $\pi_{st}$ is, up to isomorphism, the only indecomposable summand of
$I_P^G (\delta)$ in which $\pi$ appears.
\end{lem}
\begin{proof}
By Theorem \ref{thm:7.2}.a every indecomposable direct summand of
$I_P^G (\delta)$ is a standard module, say $I_Q^G (\tau \otimes \nu)$.
Let $\pi$ be a subquotient of $I_Q^G (\tau \otimes \nu)$.
By definition we have equalities
\[
\mc N (\pi) = \mc N (\pi_{st}) = \mc N (I_P^G (\delta)) =
\mc N (\mc L (Q,\tau \otimes \nu) ).
\]
Lemma \ref{lem:7.3} shows that $\pi$ must be the irreducible quotient of
$I_Q^G (\tau \otimes \nu)$. Then $I_Q^G (\tau \otimes \nu)$ is a
standard module with quotient $\pi$, so by Theorem \ref{thm:7.1}.c
$I_Q^G (\tau \otimes \nu)$ is isomorphic to $\pi_{st}$.
\end{proof}

Next we generalize Lemma \ref{lem:7.4} to not necessarily positive
induction data.

\begin{thm}\label{thm:7.5}
Let $\pi \in \Irr (\C G)$. There exists an induction datum
$(P,L,\delta)$ and $\kappa \in \Irr \, \C [R_\delta,\natural_\delta]$,
unique up to $G$-association, such that:
\begin{itemize}
\item $\pi$ is a constituent of $I_P^G (\delta)_\kappa$,
\item $\mc N (I_P^G (\delta))$ is maximal for the previous property.
\end{itemize}
Moreover, in this case $\mc N (\pi) = \mc N (I_P^G (\delta))$.
\end{thm}
\begin{proof}
Without $\kappa$, this is a reformulation of \cite[Theorem 2.15]{Sol1}.
The additional claims about $\kappa$ follow from Lemma \ref{lem:7.4}
and \eqref{eq:7.4}.
\end{proof}

We are ready to characterize standard modules without temperedness
or positivity. We abbreviate the previous $\tau \otimes \nu$ to $\tau'$.

\begin{thm}\label{thm:7.6}
Let $I_Q^G (\tau')$ be a quasi-standard $\C G$-module which has
a unique irreducible quotient $\pi$ and satisfies $\mc N (I_Q^G (\tau'))
= \mc N (\pi)$. Then $I_Q^G (\tau')$ is a standard $\C G$-module.
\end{thm}
\begin{proof}
By Theorem \ref{thm:7.2}, there exists a positive induction datum
$(P,L,\delta)$ and $\kappa \in \Irr \, \C [R_\delta,\natural_\delta]$
such that $\pi_{st} \cong I_P^G (\delta )_\kappa$. By the definition of
quasi-standard, there exists an induction datum $(P',L',\omega)$ and
$\kappa' \in \Irr \, \C [R_\omega, \natural_\omega]$ such that
$I_Q^G (\tau') \cong I_{P'}^G (\omega )_{\kappa'}$. The condition
$\mc N (I_Q^G (\tau')) = \mc N (\pi)$ and Theorem \ref{thm:7.5} imply that
$I_Q^G (\tau')$ has maximal $\mc N$-value among the quasi-standard
modules with $\pi$ as constituent. As $\mc N (\pi) = \mc N (\pi_{st})$,
the same holds for $I_P^G (\delta)_\kappa$.
By the uniqueness in Theorem \ref{thm:7.5},
$(P,L,\delta,\kappa)$ and $(P',L',\omega,\kappa')$ are $G$-associate.
By Lemma \ref{lem:7.7}, there exists a nonzero $G$-intertwining operator
\[
J : I_P^G (\delta)_\kappa \to I_{P'}^G (\omega)_{\kappa'} .
\]
Let $q : I_P^G (\delta)_\kappa \cong \pi_{st} \to \pi$ and
$q' : I_{P'}^G (\omega)_{\kappa'} \cong I_Q^G (\tau') \to \pi$ be the
quotient maps. The kernel of $J$ is not the whole of $I_P^G (\delta)_\kappa
\cong \pi_{st}$, so it is contained in $\ker q$ (because $\pi$ is the
unique irreducible quotient of $\pi_{st}$). Thus $J$ induces an injection
\[
\pi \cong I_P^G (\delta)_\kappa / \ker q \xrightarrow{\bar J}
I_Q^G (\omega)_{\kappa'} / J (\ker q).
\]
By Lemmas \ref{lem:7.3} and \ref{lem:7.7}, $\pi$ appears with multiplicity one 
in $I_P^G (\delta)_\kappa$ and in $I_Q^G (\omega)_{\kappa'}$. Since
$\bar J$ is injective, $J(\ker q)$ does not contain $\pi \cong \bar J (\pi)$ 
as subquotient. Hence $J(\ker q) \subset \ker q'$ and $q' \circ \bar J$ 
is nonzero. It follows that
the image of $J$ is a subrepresentation of $I_Q^G (\omega)_{\kappa'}$ not
contained in $\ker q'$. As $\pi$ is the unique irreducible quotient of
$I_Q^G (\omega)_{\kappa'}$, $J$ is surjective. Further
$I_P^G (\delta)_\kappa$ and $I_Q^G (\omega)_{\kappa'}$ have the same
Jordan--H\"older content (Lemma \ref{lem:7.7}), so $J$ is an isomorphism.
\end{proof}

\section{The action of the Galois group on representations} \label{sec:3}

Let $\Gal (\C / \Q)$ be the automorphism group of the field extension $\C / \Q$.
Strictly speaking this is not a Galois extension because it is not algebraic,
but for brevity we still speak of the Galois group of this extension.

For $\gamma \in \Gal (\C / \Q)$, let $\C_\gamma$ be $\C$ as $\C$-$\C$-bimodule
with action
\[
z_1 \cdot v \cdot z_2 = z_1 v \gamma (z_2) \qquad z_i \in \C, v \in \C_\gamma .
\]
For a $G$-representation $(\pi,V_\pi)$ we define 
${}^\gamma V_\pi = \C_\gamma \otimes_\C V_\pi$. This means that as 
an abelian group ${}^\gamma V_\pi$ can identified with $V_\pi$, but with the
adjusted scalar multiplication 
\[
z (1 \otimes v) = z \otimes v = 1 \otimes \gamma^{-1}(z) v 
\qquad z \in \C, v \in V_\pi .
\]
\begin{defn}
$\gamma \cdot \pi$ is the $G$-representation on ${}^\gamma V_\pi$ given by 
\[
(\gamma \cdot \pi)(g) (z \otimes v) = z \otimes \pi (g) v \qquad
g \in G, z \in \C_\gamma, v \in V_\pi .
\]
\end{defn}

If $\lambda$ lies in the dual space $V_\pi^*$, then $z \otimes v \mapsto
z \gamma (\lambda (v))$ lies in $({}^\gamma V_\pi)^*$.  
For a matrix coefficient $m_{v,\lambda} : g \mapsto \lambda (\pi (g) v)$ 
of $\pi$, the corresponding matrix coefficient of $\gamma \cdot \pi$ is
$g \mapsto \gamma \big( \lambda (\pi (g) v) \big)$. Thus, for any finite
dimensional representation $\pi'$, $\gamma \cdot \pi'$ can be obtained from 
$\pi'$ by applying $\gamma$ to the matrices that define $\pi'$.

The action of $\Gal (\C / \Q)$ on $G$-representations preserves irreducibility
and cus\-pi\-dality \cite[Theorem 2.3.(1)]{KSV}. In general it does not preserve
unitarity or temperedness, as can already be seen in the case $G = GL_1 (F)$.

It is easy to check that the action of $\Gal (\C / \Q)$ on representations
of $G$ or $L$ commutes with unnormalized parabolic induction. However, that
is not entirely true for normalized parabolic induction. Consider the 
modular function $\delta_P$ of $P = L U_P$. It takes values in 
$q_F^\Z$, where $q_F$ denotes the cardinality of the residue field of $F$.
In particular $\Gal (\C / \Q)$ fixes $\delta_P$. But $\delta_P^{1/2}$ takes
values in $(q_F^{1/2})^\Z$, and if $q_F^{1/2} \notin \Q$, then some
elements of $\Gal (\C / \Q)$ send $q_F^{1/2}$ to $-q_F^{1/2}$. It follows that for
every $\gamma \in \mr{Gal}(\C / \Q)$ there exists a unique quadratic character
$\epsilon_{P,\gamma} : L = P / U_P \to \{ \pm 1 \}$ such that
\[
\gamma \cdot \delta_P^{1/2} = \delta_P^{1/2} \otimes \epsilon_{P,\gamma}.
\]

\begin{prop}\label{prop:3.1}
\textup{\cite{KSV}} 
\enuma{
\item As a character of $L$, $\epsilon_{P,\gamma}$ depends on $L$ and $G$,
but not on the choice of the parabolic subgroup $P$ with Levi factor $L$.
\item The group $N_G (L)$ fixes $\epsilon_{P,\gamma}$.
\item For any $L$-representation $\pi$, there is an isomorphism $\gamma 
\cdot I_P^G (\pi) \cong I_P^G (\gamma \cdot \pi \otimes \epsilon_{P,\gamma})$.
\item For any finite length $L$-representation $(\pi ,V_\pi)$:
\[
J_{P'|P}(\gamma \cdot \pi \otimes \epsilon_{P,\gamma}) (z \otimes v) =
z \otimes J_{P'|P}(\pi)(v) \qquad z \in \C_\gamma, v \in I_P^G (V_\pi) .
\]
\item For any $\pi \in \Irr (L)$, $\mu_{G,L}(\gamma \cdot \pi \otimes 
\epsilon_{P,\gamma}) = \prod_{\alpha \in \Phi (G,A_L)^+} \mu_{L_\alpha,L}
(\gamma \cdot \pi \otimes \epsilon_{L U_\alpha,\gamma})$. 
}
\end{prop}
\begin{proof}
(a) and (b) are \cite[Lemma 5.11]{KSV} and (c) is \cite[(5.12)]{KSV}.\\
(d) This follows directly from part (c) and the definitions of
$\gamma \cdot I_P^G (\pi)$ and $J_{P'|P}$.\\
(e) Part (d) and the definition of $\mu_{G,L}$ in 
\eqref{eq:6.1}--\eqref{eq:1.3} show that
\begin{equation}\label{eq:3.7}
\begin{aligned}
\mu_{G,L}(\gamma \cdot \pi \otimes \epsilon_{P,\gamma}) & =
J_{P|\bar P} ( \gamma \cdot \pi \otimes \epsilon_{P,\gamma} ) \circ
J_{\bar P |P} ( \gamma \cdot \pi \otimes \epsilon_{P,\gamma} ) \\
& = \gamma \big(  J_{P|\bar P} ( \pi ) \circ 
J_{\bar P |P} ( \pi) \big) = \gamma (\mu_{G,L}(\pi)) .
\end{aligned}
\end{equation}
From \eqref{eq:1.4} and \eqref{eq:3.7} we deduce
\begin{multline*}
\mu_{G,L}(\gamma \cdot \pi \otimes \epsilon_{P,\gamma}) = \gamma (\mu_{G,L}(\pi)) \\
= \gamma \Big( \prod_{\alpha \in \Phi (G,A_L)^+} \mu_{L_\alpha,L} (\pi) \Big) =
\prod_{\alpha \in \Phi (G,A_L)^+} \mu_{L_\alpha,L}
(\gamma \cdot \pi \otimes \epsilon_{L U_\alpha,\gamma}) . \qedhere
\end{multline*}
\end{proof}

\subsection{The Galois action on quasi-standard modules} \

We would like to understand how Gal$(\C / \Q)$ acts on quasi-standard
$\C G$-modules. A crucial step is the following result, which for semisimple
groups over $p$-adic fields is due to Clozel (unpublished).

\begin{thm}\label{thm:3.2}
\textup{\cite[Theorem 4.6]{KSV}} \\
Let $\delta$ be an essentially square-integrable $L$-representation and let
$\gamma \in \Gal (\C / \Q)$. Then $\gamma \cdot \delta$ is also
essentially square-integrable.
\end{thm}

We preserve the notations from Theorem \ref{thm:3.2}. Recall from 
\eqref{eq:5.5} and \eqref{eq:7.9} that for $r \in R_\delta$ we have 
$J_\delta (r) \in \End_{\C G} (I_P^G (\delta))$, and that these operators
span a twisted group algebra $\C [R_\delta, \natural_\delta]$. 
Theorem \ref{thm:3.2} tells us that $\gamma \cdot \delta \otimes 
\epsilon_{P,\gamma}$ is essentially square-integrable. 
By \cite[Proposition 5.12]{KSV} we have 
$R_{\gamma \delta \otimes \epsilon_{P,\gamma}} = R_\delta$, and we may define
\begin{align*}
& J_{\gamma \delta \otimes \epsilon_{P,\gamma}}(r) \in \End_{\C G} \big(
I_P^G (\gamma \cdot \pi \otimes \epsilon_{P,\gamma}) \big) =
\End_{\C G} \big( \C_\gamma \otimes_{\C,\gamma} I_P^G (\delta) \big) \\
& J_{\gamma \delta \otimes \epsilon_{P,\gamma}}(r) (z \otimes v) =
z \otimes J_\delta (r) (v) .
\end{align*}
By construction $J_\delta (r) J_\delta (r') = \natural_\delta (r,r')
J_\delta (r r')$, which implies that 
\[
J_{\gamma \delta \otimes \epsilon_{P,\gamma}}(r) J_{\gamma \delta \otimes 
\epsilon_{P,\gamma}}(r') = \gamma (\natural_\delta (r,r'))
J_{\gamma \delta \otimes \epsilon_{P,\gamma}}(r r') .
\]
Hence the operators $J_{\gamma \delta \otimes \epsilon_{P,\gamma}}(r)$ span a
twisted group algebra 
\begin{equation}\label{eq:3.1}
\C [R_{\gamma \delta \otimes \epsilon_{P,\gamma}}, \natural_{\gamma \delta 
\otimes \epsilon_{P,\gamma}}] = \C [R_\delta, \gamma \natural_\delta] .
\end{equation}
It is easily seen that there is a canonical bijection 
\[
\Irr (\C [R_\delta,\natural_\delta]) \to \Irr (\C [R_\delta,\gamma 
\natural_\delta]) : \kappa \mapsto \gamma \cdot \kappa .
\]

\begin{lem}\label{lem:7.9}
For any quasi-standard $\C G$-module $I_P^G (\delta )_\kappa$ and any 
$\gamma \in \Gal (\C / \Q)$, there is an isomorphism
\[
\gamma \cdot I_P^G (\delta )_\kappa \cong I_P^G \big( \gamma \cdot \delta 
\otimes \epsilon_{P,\gamma} \big)_{\gamma \cdot \kappa} .
\]
In particular the action of $\Gal (\C / \Q)$ on $\Rep (G)$ 
stabilizes the set of quasi-standard $\C G$-modules.
\end{lem}
\begin{proof}
One step in the construction of $I_P^G (\delta )_\kappa$ is the representation 
$I_{L_\delta \cap P}^{L_\delta} (\delta)_\kappa$ from \eqref{eq:7.2}, to get
$I_P^G (\delta)_\kappa$ we parabolically induce that. Recall that both parabolic
induction and its normalized version are transitive \cite[Lemme VI.1.4]{Ren}, and 
that an ingredient for the latter is the equality of modular functions
$\delta_P = \delta_{L_\delta P} \delta_{L_\delta \cap P}$. This equality entails that
\begin{equation}\label{eq:7.16}
\epsilon_{P,\gamma} = 
\epsilon_{L_\delta P,\gamma} \, \epsilon_{L_\delta \cap P,\gamma}.
\end{equation}
With Proposition \ref{prop:3.1}.c and \eqref{eq:3.1} we compute
\begin{equation}\label{eq:7.10}
\gamma \cdot I_P^G (\delta)_\kappa \cong I_{L_\delta P}^G \big( \gamma \cdot 
I_{L_\delta \cap P}^{L_\delta} (\delta)_\kappa \otimes \epsilon_{P,\gamma} \big) 
\cong I_{L_\delta P}^G \big( I_{L_\delta \cap P}^{L_\delta} (\gamma \cdot \delta 
\otimes \epsilon_{L_\delta \cap P,\gamma} )_{\gamma \cdot \kappa} \otimes 
\epsilon_{L_\delta P,\gamma} \big) .
\end{equation}
Notice that these expressions are well-defined because $\gamma \cdot \delta 
\otimes \epsilon_{L_\delta \cap P,\gamma}$ is essentially square-integrable
(Theorem \ref{thm:3.2}).
Tensoring by $\epsilon_{L_\delta P,\gamma} \in X_\nr (L_\delta)$ commutes with all 
the operations involved in $I_{L_\delta \cap P}^{L_\delta} (\gamma \cdot \delta 
\otimes \epsilon_{L_\delta \cap P,\gamma} )_{\gamma \cdot \kappa}$. By that and 
\eqref{eq:7.16}, the right hand side of \eqref{eq:7.10} is isomorphic to
\begin{equation}\label{eq:7.15}
I_{L_\delta P}^G \big( I_{L_\delta \cap P}^{L_\delta} (\gamma \cdot \delta
\otimes \epsilon_{P,\gamma} )_{\gamma \cdot \kappa} \big) \cong
I_{L_\delta P}^G \big( \Hom_{\C [R_\delta, \gamma \natural_\delta]}
\big( \gamma \cdot \kappa , I_{L_\delta \cap P}^{L_\delta} (\gamma \cdot \delta 
\otimes \epsilon_{P,\gamma} ) \big) \big) .
\end{equation}
By the transitivity of normalized parabolic induction, the right hand side of 
\eqref{eq:7.15} equals the quasi-standard $\C G$-module $I_P^G \big( \gamma 
\cdot \delta \otimes \epsilon_{P,\gamma} \big)_{\gamma \cdot \kappa}$.
\end{proof}

Recall $\mc N$ from \eqref{eq:7.5}.
Although $\gamma \cdot \delta \otimes \epsilon_{P,\gamma}$ is essentially
square-integrable and $\mc N (\gamma \cdot \delta 
\otimes \epsilon_{P,\gamma} ) = \mc N (\gamma \cdot \delta)$, it is not 
obvious whether $\mc N (\gamma \cdot \delta)$ equals $\mc N (\delta)$
for all $\gamma \in \Gal (\C / \Q)$. Via Theorem \ref{thm:6.7} and Lemma
\ref{lem:6.10}, that can be reduced to the question:
\begin{equation}\label{eq:3.3}
\text{are the numbers } q_\alpha^2, q_{\alpha *}^2 \text{ from Theorem
\ref{thm:6.2} always rational?}
\end{equation}
In \cite{Oha,SolParam} it has been shown that $q_\alpha, q_{\alpha *}$
belong to $(q_F^{1/2})^\Z$ in the large majority of all cases. Nevertheless
there is no general proof for \eqref{eq:3.3}. This means that currently
it is known that $\mc N (\gamma \cdot \delta) = \mc N (\delta)$ for most
essentially square-integrable representations, but at the same time that
remains open in general.

\begin{prop}\label{prop:3.3}
Assume that the action of $\Gal (\C / \Q)$ preserves the $\mc N$-values
of all essentially square-integrable representations of Levi subgroups
of $G$. Then $\Gal (\C / \Q)$ stabilizes the set of standard $\C G$-modules.
\end{prop}
\begin{proof}
Consider any quasi-standard $\C G$-module $I_P^G (\delta )_\kappa$. By
\eqref{eq:7.5} and the assumptions:
\begin{multline}\label{eq:3.4}
\mc N \big( I_P^G (\delta )_\kappa \big) = \mc N (\delta) 
= \mc N (\gamma \cdot \delta) = \\ 
\mc N ( \gamma \cdot \delta \otimes \epsilon_{P,\gamma}) =  
\mc N \big( I_P^G (\gamma \cdot \delta \otimes \epsilon_{P,\gamma}) \big) =
\mc N \big( \gamma \cdot I_P^G (\delta )_\kappa \big) .
\end{multline}
Next we consider any standard $\C G$-module $\pi_{st}$, with irreducible
quotient $\pi$. We know from Lemma \ref{lem:7.9} that $\gamma \cdot \pi_{st}$
is a quasi-standard $\C G$-module. By \eqref{eq:3.4} and Theorem \ref{thm:7.5}
we have
\begin{equation}\label{eq:3.5}
\mc N (\gamma \cdot \pi) \geq \mc N ( \gamma \cdot \pi_{st}) = \mc N (\pi_{st})
= \mc N (\pi) .
\end{equation}
We may als apply this to $\gamma^{-1}$ acting on $\gamma \cdot \pi$, then we find
\begin{equation}\label{eq:3.6}
\mc N (\pi) = \mc N (\gamma^{-1} \cdot \gamma \cdot \pi) \geq
\mc N (\gamma \cdot \pi) \geq \mc N (\pi). 
\end{equation}
We conclude that $\mc N (\gamma \cdot \pi)$ equals $\mc N (\pi)$.

From \eqref{eq:3.5} and \eqref{eq:3.6} we see that $\mc N (\gamma \cdot \pi_{st})
= \mc N (\gamma \cdot \pi)$. As $\pi$ is a quotient of $\pi_{st}$, $\gamma
\cdot \pi$ is a quotient of $\gamma \cdot \pi_{st}$. Now we are in the right
position to apply Theorem \ref{thm:7.6}, which guarantees that $\gamma \cdot 
\pi_{st}$ is a standard $\C G$-module.
\end{proof}

\subsection{The Galois action on standard modules} \

We proceed to establish an unconditional version of Proposition \ref{prop:3.3}.
Let $(P,L,\delta)$ be a positive induction datum and let 
$\kappa \in \C [R_\delta,\natural_\delta]$. Recall from Theorem \ref{thm:7.2}
that $I_P^G (\delta )_\kappa$ is a standard $\C G$-module and that every
standard $\C G$-module has this form. Let $M \subset L$ be a Levi subgroup
and let $\sigma \in \Irr (M)$ be such that $(M,\sigma)$ represents the cuspidal
support of $(\delta, V_\delta)$.

We write $\delta = \delta_u \otimes \nu_\delta$ where $\delta_u \in \Irr (L)$
is square-integrable modulo center and $\nu_\delta \in \Hom (L,\R_{>0})$.
We note that $\nu_\delta$ is determined by $\nu_\delta |_{A_L} =
| cc(\delta) |_{A_L}$. Similarly we write $\sigma = \sigma_u \otimes \nu_\sigma$
with $\sigma \in \Irr (M)$ unitary supercuspidal and $\nu_\sigma \in 
\Hom (M,\R_{>0})$. Then $\nu_\sigma$ decomposes as $(\nu_\sigma \nu_\delta^{-1} |_M)
\, \nu_\delta |_M$ where $\nu_\delta |_M$ is trivial on $M \cap L_\der$ and 
$\nu_\sigma \nu_\delta^{-1} |_M$ is trivial on $Z(L)$. 
For $\gamma \in \Gal (\C / \Q)$ we have
\[
\gamma \cdot \delta = (\gamma \cdot \delta )_u \otimes \nu_{\gamma \delta}
\quad \text{with} \quad \nu_{\gamma \delta} |_{A_L} = 
|cc (\gamma \cdot \delta)|_{A_L} = |\gamma \cdot cc (\delta) |_{A_L} ,
\]
and analogously for $\sigma$. Moreover 
\begin{equation}\label{eq:3.8}
\nu_{\gamma \sigma} \nu_{\gamma \delta}^{-1} |_M \text{ is trivial on } Z(L) 
\quad \text{and} \quad \nu_{\gamma \delta} |_M \text{ is trivial on } M \cap L_\der,
\end{equation}
However, in general $(\gamma \cdot \delta)_u \not\cong \gamma \cdot \delta_u$ and
$(\gamma \cdot \sigma)_u \not\cong \gamma \cdot \sigma_u$.

\begin{lem}\label{lem:3.4}
Let $\alpha \in \Phi (G,A_M)$.
\enuma{
\item If $\mu_{M_\alpha,M}(\sigma) = \infty$, then $\mu_{M_\alpha,M}(\gamma \cdot 
\sigma \otimes \epsilon_{M U_\alpha,\gamma}) = \infty$ and 
\[
\langle \alpha^\vee, \log (\nu_\sigma ) \rangle 
\langle \alpha^\vee, \log (\nu_{\gamma \sigma} ) \rangle > 0 .
\]
\item If $\mu_{M_\alpha,M}(\sigma) = 0$, then $\mu_{M_\alpha,M}(\gamma \cdot 
\sigma \otimes \epsilon_{M U_\alpha,\gamma}) = 0$ and
\[
\langle \alpha^\vee, \log (\nu_\sigma ) \rangle =
\langle \alpha^\vee, \log (\nu_{\gamma \sigma} ) \rangle = 0 .
\] 
}
\end{lem}
\begin{proof}
(a) Proposition \ref{prop:3.1}.e says that $\mu_{M_\alpha,M}(\gamma \cdot 
\sigma \otimes \epsilon_{M U_\alpha,\gamma}) = \infty$. By Theorem \ref{thm:3.2},
\[
\gamma \cdot I_Q^G (\sigma) \cong I_Q^G (\gamma \cdot \sigma \otimes 
\epsilon_{M U_\alpha,\gamma}) = I_Q^G \big( (\gamma \cdot \sigma)_u \otimes 
\epsilon_{M U_\alpha,\gamma} \otimes \nu_{\gamma \sigma} \big) 
\]
has the irreducible essentially square-integrable subquotient $\gamma \cdot \delta$.
More precisely, $\gamma \cdot \delta$ is a quotient (resp. a subrepresentation) if
and only if $\delta$ is a quotient (resp. a subrepresentation) of $I_Q^G (\sigma)$.
Now Proposition \ref{prop:2.1} says that $\langle \alpha^\vee, (\log \nu_\sigma ) \rangle$
and $\langle \alpha^\vee, \log (\nu_{\gamma \sigma} ) \rangle$ have the same sign
(which is nonzero by Theorem \ref{thm:6.2}.b).\\
(b) Theorem \ref{thm:6.2} shows that $\langle \alpha^\vee, \log (\nu_\sigma ) 
\rangle = 0$. From Proposition \ref{prop:3.1}.d we see that
\[
\mu_{M_\alpha,M}\big( (\gamma \cdot \sigma)_u \otimes \epsilon_{M U_\alpha,\gamma} 
\otimes \nu_{\gamma \sigma} \big) = \mu_{M_\alpha,M}(\gamma \cdot 
\sigma \otimes \epsilon_{M U_\alpha,\gamma}) = 0. 
\]
Then Theorem \ref{thm:6.2} proves also that 
$\langle \alpha^\vee, \log (\nu_{\gamma \sigma} ) \rangle = 0$.
\end{proof}

Let $Q \subset G$ be a parabolic subgroup with Levi factor $M$, such that
$P = QL$. Recall that $\Phi (G,A_M,X_\nr (M) \sigma)$ is the set of $\alpha \in 
\Phi (G,A_M)$ for which $\mu_{M_\alpha,M}$ is not constant on $X_\nr (M)
\sigma$ (or equivalently has a zero on $X_\nr (M) \sigma$). By 
\cite[Proposition 1.3]{Hei2}, this is a reduced root system in
$X^* (A_M)$. The same holds with $L$ instead of $G$, the crucial point
is that $\sigma$ is cuspidal.

The Weyl group $W \big(\Phi (L,A_M,X_\nr (M)\sigma) \big)$ is contained
in $N_L (M) / M$ and acts on $\Irr (M)$. For any $w_\sigma \in W \big( \Phi
(L,A_M,X_\nr (M)\sigma) \big)$, $I_{Q \cap L}^L (w_\sigma \cdot \sigma)$ has
the same irreducible subquotients as $I_{Q \cap L}^L (\sigma)$, in
particular $\delta$. Furthermore $w_\sigma \cdot \nu_\delta = \nu_\delta$
because $(w_\sigma \cdot \nu_\delta )|_{Z(L)} = \nu_\delta |_{Z(L)}$.

We pick $w_{\sigma}$ such that $w_{\sigma} \cdot \sigma$ lies in
the (closed) positive Weyl chamber for $\Phi (L,A_M,X_\nr (M) \sigma)$,
with respect to the positive roots from $Q \cap L$. Next we replace 
$\sigma$ by $w_\sigma \cdot \sigma$, which is allowed because our main
interest is not $\sigma$ but $\delta$. Thus
\begin{equation}\label{eq:3.9}
\log (\nu_\sigma) \text{ is positive with respect to }
\Phi (L,A_M,X_\nr (M) \sigma) \cap \Phi (U_Q,A_M) ,
\end{equation}
but maybe not with respect to other elements of $\Phi (U_Q,A_M)$.

Recall from Proposition \ref{prop:3.1}.c that
\[
\gamma \cdot I_Q^G (\sigma) \cong I_Q^G (\gamma \cdot \sigma \otimes
\epsilon_{Q,\gamma}) = I_Q^G \big( (\gamma \cdot \sigma)_u \otimes
\epsilon_{Q,\gamma} \otimes \nu_{\gamma \sigma} \big) .
\]
By \eqref{eq:3.7} we have 
\[
\Phi \big( G,A_M, X_\nr (M) (\gamma \cdot \sigma \otimes \epsilon_{Q,\gamma})
\big) = \Phi (G,A_M, X_\nr (M) \sigma) .
\]
We pick a set of positive roots $\Phi (G,A_M,X_\nr (M) \sigma )^{'+}$
such that $\log (\nu_{\gamma \sigma}) \in X^* (A_M) \otimes_\Z \R$ 
lies in the corresponding (closed) positive Weyl chamber. 

For $\alpha \in \Phi (G,A_M,X_\nr (M) \sigma )^{'+}$ we find, using the
definition of $(\alpha |_{A_L})^\vee$ from \eqref{eq:2.5}
\begin{equation}\label{eq:3.10}
\langle \alpha^\vee, \log (\nu_{\gamma \sigma}) \rangle \geq 
\langle (\alpha |_{A_L})^\vee, \log (\nu_{\gamma \sigma}) \rangle =
\langle (\alpha |_{A_L})^\vee, \log (\nu_{\gamma \delta}) \rangle \geq 0 . 
\end{equation}
This enables us to extend $\Phi (G,A_M,X_\nr (M) \sigma )^{'+}$ to
a set of positive roots $\Phi (G,A_M )^{'+}$ of $\Phi (G,A_M)$ such that
\begin{enumerate}[(i)]
\item if $\alpha \in \Phi (G,A_M )^{'+} \cap \Phi (L,A_M )$, then
$\langle \alpha^\vee, \log (\nu_{\gamma \sigma}) \rangle \geq 0$,
\item if $\alpha \in \Phi (G,A_M )^{'+}, \alpha \notin \Phi (L,A_M )$,
then $\langle (\alpha |_{A_L} )^\vee, \log (\nu_{\gamma \delta}) \rangle 
\geq 0$.
\end{enumerate}
Let $Q' \subset G$ be the parabolic subgroup with Levi factor $M$ and
\begin{equation}\label{eq:3.11}
\Phi (U_{Q'}, A_M) = \Phi (G,A_M )^{'+} .
\end{equation}
Then (ii) says that
\begin{equation}\label{eq:3.12}
\gamma \cdot \delta \text{ and } \log (\nu_{\gamma \delta})
\text{ are positive with respect to } Q' L .
\end{equation}

\begin{lem}\label{lem:3.5}
$J'_{Q'L | QL}(\delta) : I_{QL}^G (\delta) \to I_{Q' L}^G (\delta)$
is an isomorphism.
\end{lem}
\begin{proof}
We take $\sigma$ and $\Phi (G,A_M,X_\nr (M) \sigma )^{'+}$ as above.
Suppose that $\alpha \in \Phi (U_Q,A_M)$ and $\mu_{M_\alpha,M}(\sigma) =
\infty$. Then $\langle \alpha^\vee, \log (\nu_\sigma ) \rangle \geq 0$ by
\eqref{eq:3.9}. Lemma \ref{lem:3.4}.a says that 
\[
\langle \alpha^\vee, \log (\nu_\sigma) \rangle > 0 \text{ and }
\langle \alpha^\vee, \log (\nu_{\gamma \sigma}) \rangle > 0.
\]
Now \eqref{eq:3.12} guarantees that
\[
\alpha \in \Phi (G,M, X_\nr (M) \sigma )^{'+} \subset 
\Phi (G,A_M )^{'+} = \Phi (U_Q,A_M) ,
\]
and in particular $\alpha \not\in \Phi (U_{\overline{Q'}},A_M)$.
As $U_{QL} \subset U_Q$ and $U_{Q' L} \subset U_{Q'}$, we find 
that $\mu_{M_\alpha,M}(\sigma) \neq \infty$ for all $\alpha \in 
\Phi (U_{QL},A_M) \cap \Phi (U_{\overline{Q'}},A_M)$.
Now Theorem \ref{thm:1.7} says that $J'_{Q'L | QL}(\delta)$ is
an isomorphism.
\end{proof}

In addition to $P = QL$, we write $P' = Q' L$ with $Q'$ as in \eqref{eq:3.11}.

\begin{thm}\label{thm:7.11}
Let $(P,L,\delta)$ be a positive induction datum and let 
$\gamma \in \Gal (\C / \Q)$.
\enuma{
\item $\gamma \cdot I_P^G (\delta) \cong I_{P'}^G (\gamma \cdot \delta \otimes
\epsilon_{P,\gamma})$ and $(P',L,\gamma \cdot \delta \otimes \epsilon_{P,\gamma})$
is a positive induction datum.
\item For any $\kappa \in \Irr ( \C [R_\delta,\natural_\delta])$, there exists
$\kappa' \in \Irr ( \C [R_\delta, \natural_{\gamma \cdot \delta} \otimes
\epsilon_{P,\gamma}])$ such that 
\[
\gamma \cdot I_P^G (\delta )_\kappa \cong
I_{P'}^G (\gamma \cdot \delta \otimes \epsilon_{P,\gamma})_{\kappa'}.
\]
\item The action on $\Gal (\C / \Q)$ on $\Rep ( G)$ stabilizes 
the set of standard $\C G$-modules.
}
\end{thm}
\begin{proof}
(a) From Proposition \ref{prop:3.1}.c we know that 
\[
\gamma \cdot I_P^G (\delta) \cong I_P^G (\gamma \cdot \delta \otimes
\epsilon_{P,\gamma}) .
\]
By Lemma \ref{lem:3.5} $J'_{P'|P}(\delta) : I_P^G (\delta) \to I_{P'}^G (\delta)$
is an isomorphism. The operator $J'_{P'|P}(\delta)$ is a normalized version of
$J_{P'|P}(\delta)$, and Proposition \ref{prop:3.1}.c entails that 
\[
J'_{P'|P}(\gamma \cdot \delta \otimes \epsilon_{P,\gamma}) (z \otimes v) =
z \otimes J'_{P' | P}(\delta) v \qquad z \in \C_\gamma, v \in I_P^G (V_\delta) .
\]
It follows that
\[
J'_{P'|P}(\gamma \cdot \delta \otimes \epsilon_{P,\gamma}) :
I_P^G (\gamma \cdot \delta \otimes \epsilon_{P,\gamma}) \to
I_{P'}^G (\gamma \cdot \delta \otimes \epsilon_{P,\gamma})
\]
is an isomorphism. From Theorem \ref{thm:3.2} we see that 
$(P',L,\gamma \cdot \delta \otimes \epsilon_{P,\gamma})$ is an induction datum,
and \eqref{eq:3.12} says that it is positive.\\
(b) Recall from \eqref{eq:7.9} and Theorem \ref{thm:7.2}.a that every indecomposable
direct summand of $I_P^G (\delta)$ is isomorphic to $I_P^G (\delta)_\kappa$
for some $\kappa \in \Irr ( \C [R_\delta,\natural_\delta])$. By part (a) that
applies also to $\gamma \cdot I_P^G (\delta)$, and with \eqref{eq:3.1} we can
simplify it a little to
\begin{equation}\label{eq:3.13}
I_{P'}^G (\gamma \cdot \delta \otimes \epsilon_{P,\gamma}) \cong 
\bigoplus\nolimits_{\kappa' \in \Irr ( \C [R_\delta, \natural_{\gamma \cdot \delta} 
\otimes \epsilon_{P,\gamma}])} \kappa' \otimes I_{P'}^G (\gamma \cdot \delta \otimes 
\epsilon_{P,\gamma} )_{\kappa'} .
\end{equation}
As $I_P^G (\delta )_\kappa$ is isomorphic to an indecomposable direct summand of
$I_P^G (\delta )$, the representation $\gamma \cdot I_P^G (\delta)_\kappa$ is
isomorphic to an indecomposable direct summand of 
\[
\gamma \cdot I_P^G (\delta) \cong 
I_{P'}^G (\gamma \cdot \delta \otimes \epsilon_{P,\gamma}) .
\]
By \eqref{eq:3.13}, the latter has the form
$I_{P'}^G (\gamma \cdot \delta \otimes \epsilon_{P,\gamma})_{\kappa'}$ for
some $\kappa'$.\\
(c) Recall from Theorem \ref{thm:7.2}.b that every standard $\C G$-module has the
form $I_P^G (\delta )_\kappa$. By parts (a) and (b) and Theorem \ref{thm:7.2}.a, 
$\gamma \cdot I_P^G (\delta)_\kappa$ is (isomorphic to) a standard $\C G$-module. 
\end{proof}

\end{document}